\documentclass[a4paper,11pt,fleqn]{article}
\usepackage{enumitem}
\usepackage{amsmath}
\usepackage{amssymb}
\usepackage{mathtools}
\usepackage{pifont}
\usepackage{theorem} 
\usepackage{euscript}
\usepackage{exscale,relsize}
\usepackage{epic,eepic}
\usepackage{multicol}
\usepackage{tikz}
\usepackage{pgfplots}
\usepackage{makeidx}
\usepackage{srcltx} 
\usepackage{url}
\usepackage{charter}
\usepackage{mathrsfs} 
\usepackage{graphicx}
\usepackage{authblk}
\usepackage{float}
\newcommand{\email}[1]{\href{mailto:#1}{\nolinkurl{#1}}}
\usepackage[normalem]{ulem} 
\newlength{\mySubFigSize}
\definecolor{labelkey}{rgb}{0,0.08,0.45}
\definecolor{refkey}{rgb}{0,0.6,0.0}
\definecolor{Brown}{rgb}{0.45,0.0,0.05}
\definecolor{dgreen}{rgb}{0.00,0.49,0.00}
\definecolor{dblue}{HTML}{0455BF}
\definecolor{orng}{HTML}{D35400}
\definecolor{dred}{HTML}{D90404}
\definecolor{Dblue}{HTML}{8602DC}
\RequirePackage[colorlinks,hyperindex]{hyperref}
\hypersetup{linktocpage=true,citecolor=dblue,linkcolor=dgreen}
\renewcommand{\leq}{\ensuremath{\leqslant}}
\renewcommand{\geq}{\ensuremath{\geqslant}}

\newcommand{\Frac}[2]{\displaystyle{\frac{#1}{#2}}} 
 
\newcommand{\scal}[2]{{\langle{{#1}\mid{#2}}\rangle}}

\newcommand{\abscal}[2]{\left|\left\langle{{#1}\mid{#2}}%
\right\rangle\right|} 
\newcommand{\menge}[2]{\big\{{#1}~\big |~{#2}\big\}}

\newcommand{\Argmin}{\ensuremath{{\text{\rm Argmin}}}}
\newcommand{\HH}{\ensuremath{{\mathcal H}}}
\newcommand{\HS}{\ensuremath{{\mathsf H}}}
\newcommand{\GG}{\ensuremath{{\mathcal G}}}

\newcommand{\II}{\ensuremath{{\mathbb I}}}
\newcommand{\Sum}{\ensuremath{\displaystyle\sum}}

\newcommand{\emp}{\ensuremath{{\varnothing}}}

\newcommand{\Id}{\ensuremath{\operatorname{Id}}}
\newcommand{\sId}{\ensuremath{\mathsf{Id}}}

\newcommand{\RR}{\ensuremath{\mathbb{R}}}

\newcommand{\RPP}{\ensuremath{\left]0,+\infty\right[}}

\newcommand{\soft}[1]{\ensuremath{{\:\text{\rm soft}}_{{#1}}\,}}
\newcommand{\hard}[1]{\ensuremath{{\:\text{\rm hard}}_{{#1}}\,}}

\newcommand{\RLX}{\ensuremath{\left[-\infty,+\infty\right[}}

\newcommand{\RX}{\ensuremath{\left]-\infty,+\infty\right]}}

\newcommand{\NN}{\ensuremath{\mathbb N}}

\newcommand{\weakly}{\ensuremath{\:\rightharpoonup\:}}

\newcommand{\ran}{\ensuremath{\text{\rm ran}\,}}

\newcommand{\pinf}{\ensuremath{{+\infty}}}
\newcommand{\minf}{\ensuremath{{-\infty}}}
\newcommand{\tv}{\ensuremath{\text{\rm tv}\,}}

\newcommand{\dom}{\ensuremath{\text{\rm dom}\,}}

\newcommand{\prox}{\ensuremath{\text{\rm prox}}}

\newcommand{\proj}{\ensuremath{\text{\rm proj}}}
\newcommand{\sproj}{\ensuremath{\text{\rm sproj}}}
\newcommand{\Fix}{\ensuremath{\text{\rm Fix}\,}}

\newcommand{\ii}{\ensuremath{\mathrm i}}

\newcommand{\card}{\ensuremath{\text{\rm card}\,}}
\newcommand{\sign}{\ensuremath{\text{\rm sign}}}

\newcommand{\infconv}{\ensuremath{\mbox{\small$\,\square\,$}}}

\newcommand{\zeroun}{\ensuremath{\left]0,1\right[}}

\def\abstract{\noindent{\bfseries Abstract}. \ignorespaces}

\newtheorem{theorem}{Theorem}[section]
\newtheorem{lemma}[theorem]{Lemma}
\newtheorem{corollary}[theorem]{Corollary}
\newtheorem{proposition}[theorem]{Proposition}

\theoremstyle{plain}{\theorembodyfont{\rmfamily}%
\newtheorem{example}[theorem]{Example}}
\theoremstyle{plain}{\theorembodyfont{\rmfamily}%
\newtheorem{remark}[theorem]{Remark}}
\theoremstyle{plain}{\theorembodyfont{\rmfamily}%
\newtheorem{algorithm}[theorem]{Algorithm}}
\theoremstyle{plain}{\theorembodyfont{\rmfamily}%
}
\theoremstyle{plain}{\theorembodyfont{\rmfamily}%
}
\theoremstyle{plain}{\theorembodyfont{\rmfamily}%
\newtheorem{definition}[theorem]{Definition}}
\theoremstyle{plain}{\theorembodyfont{\rmfamily}%
\newtheorem{problem}[theorem]{Problem}}

\numberwithin{equation}{section}
\begin{document}

\title{\sffamily\huge Reconstruction of Functions from \\
Prescribed Proximal Points\thanks{Contact 
author: P. L. Combettes, {\email{plc@math.ncsu.edu}},
phone:+1 (919) 515 2671. The work of P. L. Combettes was 
supported by the National Science Foundation under grant 
CCF-1715671 and the work of Z. C. Woodstock was supported by 
the National Science Foundation under grant DGE-1746939.}}

\author{Patrick L. Combettes and Zev C. Woodstock\\
\small North Carolina State University,
Department of Mathematics,
Raleigh, NC 27695-8205, USA\\
\small \email{plc@math.ncsu.edu}\,, \email{zwoodst@ncsu.edu}\\
}

\date{{~}}
\maketitle

\centerline{\emph{Dedicated to the memory of Noli N. Reyes
(1963--2020)}}

\bigskip

\begin{abstract} 
Under investigation is the problem of finding the best
approximation of a function in a Hilbert space subject to convex
constraints and prescribed nonlinear transformations. We show that
in many instances these prescriptions can be represented using
firmly nonexpansive operators, even when the original observation
process is discontinuous. The proposed framework thus captures a
large body of classical and contemporary best approximation
problems arising in areas such as harmonic analysis, statistics,
interpolation theory, and signal processing. The resulting problem
is recast in terms of a common fixed point problem and solved with
a new block-iterative algorithm that features approximate
projections onto the individual sets as well as an extrapolated
relaxation scheme that exploits the possible presence of affine
constraints. A numerical application to signal recovery is
demonstrated. 
\end{abstract}


\newpage
\section{Introduction}
\label{sec:1}

Let $\HH$ be a real Hilbert space with scalar product
$\scal{\cdot}{\cdot}$ and associated norm $\|\cdot\|$, let
$x_0\in\HH$, let $U$ and $V$ be closed vector subspaces of
$\HH$ with projection operators $\proj_U$ and $\proj_V$,
respectively, and let $p\in V$. The basic best approximation
problem
\begin{equation}
\label{e:1}
\text{minimize}\:\;\|x-x_0\|\quad\text{subject to}\quad
x\in U\quad\text{and}\quad\proj_V x=p
\end{equation}
covers a wide range of scenarios in areas such as harmonic
analysis, signal processing, interpolation theory, and optics
\cite{Noli18,Joat10,Goss11,Jami07,Char20,Mele96,Mont82,%
Reye13,Youl78}.
In this setting, a function of interest $\overline{x}\in\HH$ is
known to lie in the subspace $U$ and its projection $p$ onto the
subspace $V$ is known. The goal of \eqref{e:1} is then to find
the best approximation to $x_0$ that is compatible with these two
pieces of information. For example, band-limited extrapolation
\cite{Papo75} aims at recovering a minimum energy
band-limited function
$\overline{x}\in\HH=L^2(\RR)$ from the knowledge of its values on
an interval $A$. This corresponds to the instance of \eqref{e:1} in
which $x_0=0$, $V$ is the subspace of functions vanishing outside
of $A$, $U$ is the subspace of functions with Fourier transform
supported by a compact interval around the origin, and
$p=1_A\overline{x}$, where $1_A$ denotes the characteristic
function of $A$. As shown in \cite{Youl78}, if \eqref{e:1} is
feasible (see \cite{Joat10} for necessary and sufficient
conditions), then the sequence $(x_n)_{n\in\NN}$ constructed by
iterating 
\begin{equation}
\label{e:youla78}
(\forall n\in\NN)\quad x_{n+1}=p+\proj_Ux_n-
\proj_V(\proj_Ux_n)
\end{equation}
converges strongly to its solution. The extension of \eqref{e:1} to
finitely many vector subspaces $(U_j)_{j\in J}$ and 
$(V_k)_{k\in K}$ investigated in \cite{Joat10} is to 
\begin{equation}
\label{e:81}
\text{minimize}\:\;\|x-x_0\|\quad\text{subject to}\:\;
x\in\bigcap_{j\in J} U_j\:\;\text{and}\:\;
(\forall k\in K)\:\;\proj_{V_k}x=p_k,\:\;\text{where}\:\;
p_k\in V_k,
\end{equation}
and it can be solved using affine projection methods.
In many applications, the constraint sets 
\cite{Chie14,Chui92,Aiep96,Sign03,Deut97,Fava40,Micc88,Mula98}
or the operators yielding the prescribed values $(p_k)_{k\in K}$
\cite{Avil17,Boch13,Fouc17,Marm20,Renc19,Teml98,Tesh19}
may not be linear. Our objective is to extend the linear
formulation \eqref{e:81} by employing closed convex constraint
subsets $(C_j)_{j\in J}$, together with prescriptions
$(p_k)_{k\in K}$ resulting from nonlinear operators 
$(F_k)_{k\in K}$, i.e., 
\begin{equation}
\label{e:80}
\text{minimize}\:\;\|x-x_0\|\quad\text{subject to}\quad
x\in\bigcap_{j\in J} C_j\quad\text{and}\quad
(\forall k\in K)\quad F_kx=p_k.
\end{equation}
In view of \eqref{e:81}, projection operators onto closed convex
sets constitute a natural class of candidates for the operators
$(F_k)_{k\in K}$. For instance, in \cite{Renc19,Stud12,Tesh19},
$F_k$ is the projection operator onto a hypercube. However, many
prescriptions $(p_k)_{k\in K}$ found in the literature, in
particular those of \cite{Boch13,Fouc17,Marm20,Teml98}, do
not reduce to best approximations from closed convex sets, and a
more general formalism must be considered to represent them. A
generalization of the notion of a best approximation was proposed
by Moreau \cite{Mor62b}, who called the \emph{proximal point} of
$\overline{x}\in\HH$ relative to a proper lower semicontinuous
convex function $f_k\colon\HH\to\RX$ the unique minimizer
$p_k\in\HH$ of the function 
\begin{equation}
\label{e:jjm1}
y\mapsto f_k(y)+\frac{1}{2}\|\overline{x}-y\|^2, 
\end{equation}
and wrote $p_k=\prox_{f_k}\overline{x}$.
This mechanism defines the proximity operator
$\prox_{f_k}\colon\HH\to\HH$ of $f_k$. The case of a
projector onto a nonempty closed convex set $D_k\subset\HH$ is
recovered by letting $f_k=\iota_{D_k}$, where
\begin{equation}
\label{e:iota}
(\forall x\in\HH)\quad\iota_{D_k}(x)=
\begin{cases}
0,&\text{if}\;\;x\in D_k;\\
\pinf,&\text{if}\;\;x\notin D_k
\end{cases}
\end{equation}
is the indicator function of $D_k$. Proximity operators were
initially motivated by applications in mechanics
\cite{Brog16,Mor63m,More66} and have become a central tool in the
analysis and the numerical solution of numerous data processing
tasks \cite{Svva20,Smms05}. We shall see later that they also model
various nonlinear observation processes. The properties of
proximity operators are detailed in \cite[Chapter~24]{Livre1},
among which is the fact that the operator $\prox_{f_k}$ can be
expressed as the resolvent of the subdifferential of $f_k$, that
is, $\prox_{f_k}=(\Id+\partial f_k)^{-1}$, where 
\begin{equation}
(\forall x\in\HH)\quad\partial f_k(x)=\menge{u\in\HH}
{(\forall y\in\HH)\;\;\scal{y-x}{u}+f_k(x)\leq f_k(y)}.
\end{equation}
As shown by Moreau \cite{More65}, the set-valued operator 
$A_k=\partial f_k$ is maximally monotone, i.e.,
\begin{equation}
\label{e:mm}
(\forall x\in\HH)(\forall u\in\HH)\quad\big[\:u\in
A_kx\quad\Leftrightarrow\quad(\forall y\in\HH)
(\forall v\in A_ky)\quad\scal{x-y}{u-v}\geq 0\:\big].
\end{equation}
This property prompted
Rockafellar \cite{Roc76a} to generalize the notion of a proximal
point as follows: given a maximally monotone set-valued operator
$A_k\colon\HH\to 2^{\HH}$, the proximal point of
$\overline{x}\in\HH$ relative to $A_k$ is the unique point
$p_k\in\HH$ such that $\overline{x}-p_k\in A_kp_k$, i.e., 
$p_k=J_{A_k}\overline{x}$, where
$J_{A_k}=(\Id+A_k)^{-1}\colon\HH\to\HH$ is the resolvent of $A_k$. 
As stated in \cite[Corollary~23.9]{Livre1}, 
a remarkable consequence of Minty's theorem \cite{Mint62}
is that an operator $F_k\colon\HH\to\HH$ is the resolvent of a
maximally monotone operator $A_k\colon\HH\to 2^{\HH}$ if and only
if it is \emph{firmly nonexpansive}, meaning that
\begin{equation}
\label{e:f}
(\forall x\in\HH)(\forall y\in\HH)\quad
\|F_kx-F_ky\|^2+\|(\Id-F_k)x-(\Id-F_k)y\|^2\leq\|x-y\|^2.
\end{equation}
In view of this equivalence, we call $p_k$ a \emph{proximal point}
of $\overline{x}\in\HH$ relative to a firmly nonexpansive operator
$F_k\colon\HH\to\HH$ if $p_k=F_k\overline{x}$. As we shall show in
Section~\ref{sec:2}, firmly nonexpansive operators constitute a
powerful device to represent a variety of nonlinear processes to
generate the prescriptions $(p_k)_{k\in K}$ in \eqref{e:80}. In
light of these considerations, we propose to investigate the
following nonlinear best approximation framework. 

\begin{problem}
\label{prob:1}
Let $x_0\in\HH$ and let $J$ and $K$ be at most countable sets such
that $J\cap K=\emp$ and $J\cup K\neq\emp$. For every $j\in J$, let
$C_j$ be a closed convex subset of $\HH$ and, for every $k\in K$,
let $p_k\in\HH$ and let $F_k\colon\HH\to\HH$ be a firmly
nonexpansive operator. Suppose that there exists 
$\overline{x}\in\bigcap_{j\in J}C_j$ such that $(\forall k\in K)$ 
$F_k\overline{x}=p_k$. The task is to 
\begin{equation}
\label{e:prob1}
\text{minimize}\:\;\|x-x_0\|\quad\text{subject to}\quad
x\in\bigcap_{j\in J}C_j\quad\text{and}\quad
(\forall k\in K)\quad F_kx=p_k.
\end{equation}
\end{problem}

In Problem~\ref{prob:1}, the function of interest lies in the
intersection of the sets $(C_j)_{j\in J}$, and its proximal points
$(p_k)_{k\in K}$ relative to firmly nonexpansive operators
$(F_k)_{k\in K}$ are prescribed. The objective is to obtain the
best approximation to a function $x_0\in\HH$ from the set of
functions which satisfy these properties. 

As noted above, the numerical solution of the linear problem
\eqref{e:81} is rather straightforward with existing projection
techniques, while characterizing the existence of solutions for
any choices of the prescribed values $(p_k)_{k\in K}$ -- the
so-called inverse best approximation property -- is a more
challenging task that was carried out in \cite{Joat10}. In the
nonlinear setting, this property is of limited interest since it
fails in simple scenarios \cite[Remark~1.2]{Joat10}. Our objectives
in the present paper are to demonstrate the far reach and the
versatility of Problem~\ref{prob:1}, and to devise an efficient and
flexible numerical method to solve it. 

The remainder of the paper consists of four sections. In
Section~\ref{sec:2}, we show the ability of our proximal point
modeling to capture a variety of observation processes arising in
practice, including some which result from discontinuous
operators. In Section~\ref{sec:3}, we propose a new
block-iterative algorithm to construct the best approximation to a
reference point from a countable intersection of closed convex
sets. The algorithm features approximate projections onto the
individual sets as well as an extrapolated relaxation scheme that
exploits the possible presence of affine subspaces in the
constraint sets $(C_j)_{j\in J}$. In Section~\ref{sec:4},
Problem~\ref{prob:1} is rephrased in terms of a common fixed point
problem and the algorithm of Section~\ref{sec:3} is used to solve
it. A numerical illustration of our framework is presented in
Section~\ref{sec:5}.

\noindent
{\bfseries Notation.} 
$\HH$ is a real Hilbert space with scalar product
$\scal{\cdot}{\cdot}$, associated norm $\|\cdot\|$, and identity
operator $\Id$. The family of all subsets of $\HH$ is denoted by
$2^{\HH}$. The expressions $x_n\weakly x$ and $x_n\to x$ denote,
respectively, the weak and the strong convergence of a sequence
$(x_n)_{n\in\NN}$ to $x$ in $\HH$. The distance function to a
subset $C$ of $\HH$ is denoted by $d_C$. $\Gamma_0(\HH)$ is the
class of all lower semicontinuous convex functions from $\HH$ to
$\RX$ which are proper in the sense that they are not identically
$\pinf$. The conjugate of $f\in\Gamma_0(\HH)$ is denoted by $f^*$
and the infimal convolution operation by $\infconv$.
The set of fixed points of an operator $T\colon\HH\to\HH$
is $\Fix T=\menge{x\in\HH}{Tx=x}$. The Hilbert direct 
sum of a family of real Hilbert spaces $(\HS_i)_{i\in\II}$ is 
denoted by $\bigoplus_{i\in\II}\HS_i$. For background on convex and
nonlinear analysis, see \cite{Livre1}.

\section{Prescribed values as proximal points}
\label{sec:2}

We illustrate the fact that the proximal model adopted in
Problem~\ref{prob:1} captures a wealth of scenarios encountered in
various areas to represent information on the ideal underlying
function $\overline{x}\in\HH$ obtained through some observation
process. We discuss firmly nonexpansive observation processes in
Section~\ref{sec:21} and cocoercive ones in 
Section~\ref{sec:22}. In Section~\ref{sec:23}, we move to more
general models in which the operators need not be Lipschitzian or
even continuous. 

\subsection{Prescriptions derived from firmly nonexpansive
operators}
\label{sec:21}
We start with an instance of a proximal point prescription arising
in a decomposition setting. 

\begin{proposition}
\label{p:7}
Let $(\HS_i)_{i\in\II}$ be an at most countable family of real
Hilbert spaces, let $\HH=\bigoplus_{i\in\II}\HS_i$, let
$\overline{x}\in\HH$, and let 
$(\overline{\mathsf{x}}_i)_{i\in\II}$ be its decomposition, i.e.,
$(\forall i\in\II)$ $\overline{\mathsf{x}}_i\in\HS_i$. For every
$i\in\II$, let $\mathsf{F}_i\colon\HS_i\to\HS_i$ be a firmly
nonexpansive operator. If $\II$ is infinite, suppose that there
exists $z=(\mathsf{\mathsf{z}}_i)_{i\in\II}\in\HH$ such that
$\sum_{i\in\II}\|\mathsf{F}_i\mathsf{z}_i-\mathsf{z}_i\|^2<\pinf$.
Set $F\colon\HH\to\HH\colon x=(\mathsf{\mathsf{x}}_i)_{i\in\II}
\mapsto (\mathsf{F}_i\mathsf{x}_i)_{i\in\II}$ and 
$p=(\mathsf{F}_i\overline{\mathsf{x}}_i)_{i\in\II}$. 
Then $p$ is the proximal point of $\overline{x}$ relative to $F$.
\end{proposition}
\begin{proof}
If $\II$ is infinite, we have
\begin{align}
(\forall x\in\HH)\quad
\dfrac{1}{3}\sum_{i\in\II}
\|\mathsf{F}_i\mathsf{x}_i\|^2
&\leq\sum_{i\in\II}
\|\mathsf{F}_i\mathsf{x}_i-\mathsf{F}_i\mathsf{z}_i\|^2+
\sum_{i\in\II}\|\mathsf{F}_i\mathsf{z}_i-\mathsf{z}_i\|^2+
\sum_{i\in\II}\|\mathsf{z}_i\|^2
\nonumber\\
&\leq\sum_{i\in\II}
\|\mathsf{x}_i-\mathsf{z}_i\|^2+
\sum_{i\in\II}\|\mathsf{F}_i\mathsf{z}_i-\mathsf{z}_i\|^2
+\|z\|^2
\nonumber\\
&=\|x-z\|^2+
\sum_{i\in\II}\|\mathsf{F}_i\mathsf{z}_i-\mathsf{z}_i\|^2
+\|z\|^2
\nonumber\\
&<\pinf.
\end{align}
This shows that, in all cases, $F$ is well defined and $p\in\HH$.
Furthermore,
\begin{align}
(\forall x\in\HH)(\forall y\in\HH)\quad\|Fx-Fy\|^2
&=\sum_{i\in\II}
\|\mathsf{F}_i\mathsf{x}_i-\mathsf{F}_i\mathsf{y}_i\|^2\nonumber\\
&\leq\sum_{i\in\II}\|\mathsf{x}_i-\mathsf{y}_i\|^2-
\sum_{i\in\II}\|(\sId-\mathsf{F}_i)\mathsf{x}_i
-(\sId-\mathsf{F}_i)\mathsf{y}_i\|^2\nonumber\\
&=\|x-y\|^2-\|(\Id-F)x-(\Id-F)y\|^2.
\end{align}
Thus, $F$ is firmly nonexpansive.
\end{proof}

\begin{corollary}
\label{c:7}
Let $(\HS_i)_{i\in\II}$ be an at most countable family of real
Hilbert spaces, let $\HH=\bigoplus_{i\in\II}\HS_i$, let
$\overline{x}\in\HH$, and let 
$(\overline{\mathsf{x}}_i)_{i\in\II}$ be its decomposition. For
every $i\in\II$, let $\mathsf{f}_i\in\Gamma_0(\HS_i)$ and, if $\II$
is infinite, suppose that 
$\mathsf{f}_i\geq 0=\mathsf{f}_i(\mathsf{0})$.
Then $p=(\prox_{\mathsf{f}_i}\overline{\mathsf{x}}_i)_{i\in\II}$ 
is a proximal point of $\overline{x}$, namely, 
$p=\prox_{f}\overline{x}$, where 
$f\colon\HH\to\RX\colon{x}=({\mathsf{x}}_i)_{i\in\II}
\mapsto\sum_{i\in\II}\mathsf{f}_i(\mathsf{x}_i)$.
\end{corollary}
\begin{proof}
We first note that $f$ is proper since the functions 
$(\mathsf{f}_i)_{i\in\II}$ are. Furthermore, we observe that, for
every $i\in\II$, the function 
$f_i\colon\HH\to\RX\colon x\mapsto\mathsf{f}_i(\mathsf{x}_i)$ 
lies in $\Gamma_0(\HH)$. We therefore derive from 
\cite[Corollary~9.4]{Livre1} that $f=\sum_{i\in\II}f_i$ is lower
semicontinuous and convex. This shows that $f\in\Gamma_0(\HH)$ and
consequently that $\prox_f$ is well defined. For every $i\in\II$,
let us introduce the firmly nonexpansive operator 
$\mathsf{F}_i=\prox_{\mathsf{f}_i}$. If $\II$
is infinite, since $\mathsf{0}$ is a minimizer of each of the
functions $(\mathsf{f}_i)_{i\in\II}$, we derive from
\cite[Proposition~12.29]{Livre1} that $(\forall i\in\II)$ 
$\prox_{\mathsf{f}_i}\mathsf{0}=\mathsf{0}$. In turn, 
the condition 
$\sum_{i\in\II}\|\mathsf{F}_i\mathsf{z}_i-\mathsf{z}_i\|^2<\pinf$
holds with $(\forall i\in\II)$ $\mathsf{z}_i=\mathsf{0}$. In view
of Proposition~\ref{p:7}, $p$ is the proximal point of
$\overline{x}$ relative to $F\colon\HH\to\HH\colon x\mapsto
(\prox_{\mathsf{f}_i}{\mathsf{x}}_i)_{i\in\II}$.
Finally, since
\begin{align}
f(\prox_f\overline{x})+\frac{1}{2}
\|\overline{x}-\prox_f\overline{x}\|^2
&=\min_{y\in\HH}\bigg({f}({y})+\frac{1}{2}
\|\overline{x}-y\|^2\bigg)\nonumber\\
&=\min_{y\in\HH}\sum_{i\in\II}
\bigg(\mathsf{f}_i(\mathsf{y}_i)+\frac{1}{2}
\|\overline{\mathsf{x}}_i-\mathsf{y}_i\|^2\bigg)\nonumber\\
&=\sum_{i\in\II}\min_{\mathsf{y}_i\in\HS_i}
\bigg(\mathsf{f}_i(\mathsf{y}_i)
+\frac{1}{2}\|\overline{\mathsf{x}}_i-\mathsf{y}_i\|^2\bigg)
\nonumber\\
&=\sum_{i\in\II}\bigg(\mathsf{f}_i(\prox_{\mathsf{f}_i}
\overline{\mathsf{x}}_i)+\frac{1}{2}\|\overline{\mathsf{x}}_i-
\prox_{\mathsf{f}_i}\overline{\mathsf{x}}_i\|^2\bigg)\nonumber\\
&=f(p)+\frac{1}{2}\|\overline{x}-p\|^2,
\end{align}
we conclude that $p=\prox_f\overline{x}$.
\end{proof}

\begin{corollary}
\label{c:8}
Suppose that $\HH$ is separable, let $(e_i)_{i\in\II}$ be an
orthonormal basis of $\HH$, and let $\overline{x}\in\HH$. For 
every $i\in\II$, let $\beta_i\in\RPP$ and let
$\varrho_i\colon\RR\to\RR$ be increasing and
$1/\beta_i$-Lipschitzian.
If $\II$ is infinite, suppose that $(\forall i\in\II)$
$\varrho_i(0)=0$. Then 
$p=\sum_{i\in\II}\beta_i\varrho_i(\scal{\overline{x}}{e_i})e_i$
is a proximal point of $\overline{x}$.
\end{corollary}
\begin{proof}
For every $i\in\II$, $\beta_i\varrho_i$ is increasing and
nonexpansive, hence firmly nonexpansive. We then deduce from
Proposition~\ref{p:7} that
$\Phi\colon\ell^2(\II)\to\ell^2(\II)\colon(\xi_i)_{i\in\II}\mapsto
(\beta_i\varrho_i(\xi_i))_{i\in\II}$ is firmly nonexpansive. Now
set $L\colon\HH\to\ell^2(\II)\colon x\mapsto
(\scal{x}{e_i})_{i\in\II}$ and $F=L^*\circ\Phi\circ L$. 
Since $\|L\|=1$, it follows from
\cite[Corollary~4.13]{Livre1} that $F$ is firmly nonexpansive. This
shows that $p=L^*(\Phi(L\overline{x}))$ is the proximal point of
$\overline{x}$ relative to $F$.
\end{proof}

\begin{example}
\label{ex:19}
In the context of Corollary~\ref{c:8}, for every $i\in\II$, let
$\omega_{i}\in[0,1]$, let $\eta_i\in\RPP$, let $\delta_i\in\RPP$,
and set $\varrho_i\colon\xi\mapsto(2\omega_i/\pi)
\text{\rm arctan}(\eta_i\,\xi)+(1-\omega_i)
\sign(\xi)(1-\text{\rm exp}(-\delta_i|\xi|))$.
Then, for every $i\in\II$, $\varrho_i$ is increasing and 
$(2\omega_i\eta_i/\pi+(1-\omega_i)\delta_i)$-Lipschitzian with
$\varrho_i(0)=0$. The resulting proximal point 
\begin{equation}
p=\sum_{i\in\II}
\dfrac{\varrho_i(\scal{\overline{x}}{e_i})}
{2\omega_i\eta_i/\pi+(1-\omega_i)\delta_i}e_i
\end{equation}
models a parallel distortion of the original signal $\overline{x}$ 
\cite[Sections~10.6 \& 13.5]{Tarr19}.
\end{example}

\begin{example}[shrinkage]
\label{ex:5}
In signal processing and statistics, a powerful idea is to
decompose a function $\overline{x}\in\HH$ in an orthonormal basis
$(e_i)_{i\in\II}$ and to transform the coefficients of the
decomposition to construct nonlinear approximations with certain
attributes such as sparsity
\cite{Cham98,Siop07,Smms05,Daub04,Demo09,Dono94,Taov00}. As noted
in \cite{Smms05}, a broad model in this context is 
\begin{equation}
\label{e:jjm2}
p=\sum_{i\in\II}\big(\prox_{\phi_i}\scal{\overline{x}}{e_i}
\big)e_i
\end{equation} 
where, for every $i\in\II$, the function 
$\phi_i\in\Gamma_0(\RR)$
satisfies $\phi_i\geq 0=\phi_i(0)$ and models prior information on
the coefficient $\scal{\overline{x}}{e_i}$. 
The problem is then 
to reconstruct $\overline{x}$ given its shrunk version $p$. For
instance, in the classical work of \cite{Dono94}, $(e_i)_{i\in\II}$
is a wavelet basis and $(\forall i\in\II)$ $\phi_i=\omega|\cdot|$,
with $\omega\in\RPP$. This yields
$p=\sum_{i\in\II}(\sign(\scal{\overline{x}}{e_i})
\max\{\abscal{\overline{x}}{e_i}-\omega,0\})e_i$. In general, 
to see that $p$ in \eqref{e:jjm2} is a proximal point of
$\overline{x}$, it suffices to apply Corollary~\ref{c:8} with,
for every $i\in\II$, $\beta_i=1$ and 
$\varrho_i=\prox_{\phi_i}$, whence $\varrho_i(0)=0$
by \cite[Proposition~12.29]{Livre1}. More precisely, 
\cite[Proposition~24.16]{Livre1} entails that
$p$ is the proximal point of $\overline{x}$ relative to 
the function $f\colon\HH\to\RX\colon x\mapsto
\sum_{i\in\II}\phi_i(\scal{x}{e_i})$.
\end{example}

\begin{example}[partitioning]
\label{ex:7}
Let $(\Omega,\EuScript{F},\mu)$ be a measure space and let 
$(\Omega_i)_{i\in\II}$ be an at most countable 
$\EuScript{F}$-partition of $\Omega$. Let us consider the
instantiation of Proposition~\ref{p:7} in which 
$\HH=L^2(\Omega,\EuScript{F},\mu)$ and, 
for every $i\in\II$, 
$\HS_i=L^2(\Omega_i,\EuScript{F}_i,\mu)$, where
$\EuScript{F}_i=\menge{\Omega_i\cap S}{S\in\EuScript{F}}$.
Let $\overline{x}\in\HH$ and $(\forall i\in\II)$ 
$\overline{\mathsf{x}}_i=\overline{x}|_{\Omega_i}$. Moreover, for
every $i\in\II$, $\phi_i$ is an even function in $\Gamma_0(\RR)$
such that $\phi_i(0)=0$ and $\phi_i\neq\iota_{\{0\}}$, and we set
$\rho_i=\text{\rm max}\,\partial\phi_i(0)$. Then we derive from
Corollary~\ref{c:7} and \cite[Proposition~2.1]{Nmtm09} that the 
proximal point of $\overline{x}$ relative to 
$f\colon x\mapsto\sum_{i\in\II}\phi_i(\|\mathsf{x}_i\|)$ is
\begin{equation}
\label{e:ow}
p=\Big(\big(\prox_{\phi_i}\|\overline{\mathsf{x}}_i\|\big)
\mathsf{u}_{\rho_i}(\overline{\mathsf{x}}_i)\Big)_{i\in\II},
\quad\text{where}\quad\mathsf{u}_{\rho_i}\colon\HS_i\to\HS_i
\colon{\mathsf{x}}_i\mapsto
\begin{cases}
{\mathsf{x}}_i/\|{\mathsf{x}}_i\|,&\text{if}
\;\;\|{\mathsf{x}}_i\|>\rho_i;\\
\mathsf{0},&\text{if}\;\;\|{\mathsf{x}}_i\|\leq\rho_i.
\end{cases}
\end{equation}
For each $i\in\II$, this process eliminates the $i$th block
$\overline{\mathsf{x}}_i$ if its norm is less than 
$\rho_i\in\RPP$. 
\end{example}

\begin{example}[group shrinkage]
\label{ex:4}
In Example~\ref{ex:7}, suppose that $\Omega=\{1,\ldots,N\}$,
$\EuScript{F}=2^\Omega$, and $\mu$ is the counting measure. Then
$\HH$ is the standard Euclidean space $\RR^N$, which is decomposed
in $m$ factors as $\RR^N=\RR^{N_1}\times\cdots\times\RR^{N_m}$, 
where $\sum_{i=1}^mN_i=N$. Now suppose that 
$(\forall i\in\II=\{1,\ldots,m\})$ $\phi_i=\rho_i|\cdot|$, where
$\rho_i\in\RPP$. Then it follows from \cite[Example~14.5]{Livre1}
that the proximal point $p$ of \eqref{e:ow} is
obtained by group-soft thresholding the vector 
$\overline{x}=(\overline{\mathsf{x}}_1,\ldots,
\overline{\mathsf{x}}_m)\in\RR^N$, that is \cite{Yuan06},
\begin{equation}
p=\bigg(
\bigg(1-\frac{\rho_1}
{\max\{\|\overline{\mathsf{x}}_1\|,\rho_1\}}\bigg)
\overline{\mathsf{x}}_1,\ldots,\bigg(1-\frac{\rho_m}
{\max\{\|\overline{\mathsf{x}}_m\|,\rho_m\}}\bigg)
\overline{\mathsf{x}}_m\bigg).
\end{equation}
\end{example}

\subsection{Prescriptions derived from cocoercive operators}
\label{sec:22}

Let us first recall that, given a real Hilbert space $\GG$ and
$\beta\in\RPP$, an operator $Q\colon\GG\to\GG$ is
$\beta$-\emph{cocoercive} if
\begin{equation}
\label{e:coco}
(\forall x\in\GG)(\forall y\in\GG)\quad
\scal{x-y}{Qx-Qy}\geq\beta\|Qx-Qy\|^2,
\end{equation}
which means that $\beta Q$ is firmly nonexpansive 
\cite[Section~4.2]{Livre1}.
In the following proposition, a proximal point is constructed from
a finite family of nonlinear observations $(q_i)_{i\in\II}$ of
linear transformations of the function $\overline{x}\in\HH$, where
the nonlinearities are modeled via cocoercive operators. Item
\ref{p:1ii} below shows that this proximal point contains the same
information as the observations $(q_i)_{i\in\II}$.

\begin{proposition}
\label{p:1}
Let $(\GG_i)_{i\in\II}$ be a finite family of real Hilbert spaces
and let $\overline{x}\in\HH$. For every $i\in \II$, let
$\beta_i\in\RPP$, let $Q_i\colon\GG_i\to\GG_i$ be
$\beta_i$-cocoercive, let $L_i\colon\HH\to\GG_i$ be a nonzero
bounded linear operator, and define $q_i=Q_i(L_i\overline{x})$. 
Set 
\begin{equation}
\label{e:1a}
\beta=\dfrac{1}{\Sum_{i\in\II}{\dfrac{\|L_i\|^2}{\beta_i}}},
\quad p=\beta\sum_{i\in\II}L_i^*q_i,
\quad\text{and}\quad
F=\beta\sum_{i\in\II}L_i^*\circ Q_i\circ L_i.
\end{equation}
Then the following hold:
\begin{enumerate}
\item
\label{p:1i}
$p$ is the proximal point of $\overline{x}$ relative to $F$.
\item
\label{p:1ii}
$(\forall x\in\HH)$ $Fx=p$ $\Leftrightarrow$
$(\forall i\in\II)$ $Q_i(L_ix)=q_i$.
\end{enumerate}
\end{proposition}
\begin{proof}
\ref{p:1i}: It is clear that $p=F\overline{x}$. In addition, the 
firm nonexpansiveness of $F$ follows from
\cite[Proposition~4.12]{Livre1}. 

\ref{p:1ii}: 
Take $x\in\HH$ such that $Fx=p$. Then $Fx=F\overline{x}$
and \eqref{e:coco} yields
\begin{align}
0
&=\dfrac{\scal{Fx-F\overline{x}}{x-\overline{x}}}{\beta}
\nonumber\\
&=\sum_{i\in\II}\scal{Q_i(L_ix)
-Q_i(L_i\overline{x})}{L_ix-L_i\overline{x}}\nonumber\\
&\geq\sum_{i\in\II}\beta_i\|Q_i(L_ix)-Q_i(L_i\overline{x})\|^2
\nonumber\\
&=\sum_{i\in\II}\beta_i\|Q_i(L_ix)-q_i\|^2,
\end{align}
and therefore $(\forall i\in\II)$ $Q_i(L_ix)=q_i$. The reverse
implication is clear.
\end{proof}

Next, we consider the case when the observations $(q_i)_{i\in\II}$
in Proposition~\ref{p:1} are obtained through proximity operators. 

\begin{proposition}
\label{p:2}
Let $(\GG_i)_{i\in\II}$ be a finite family of real Hilbert spaces
and let $\overline{x}\in\HH$. For every $i\in\II$, let 
$g_i\in\Gamma_0(\GG_i)$, let $L_i\colon\HH\to\GG_i$ be a nonzero
bounded linear operator, and define 
$q_i=\prox_{g_i}(L_i\overline{x})$. Suppose that 
$\beta=1/(\sum_{i\in\II}{\|L_i\|^2})$, and set
$p=\beta\sum_{i\in\II}L_i^*q_i$ and
$F=\beta\sum_{i\in\II}L_i^*\circ\prox_{g_i}\circ L_i$.
Then the following hold:
\begin{enumerate}
\item
\label{p:2ii}
$p$ is the proximal point of $\overline{x}$ relative to $F$.
\item
\label{p:2iii}
$(\forall x\in\HH)$ $Fx=p$ $\Leftrightarrow$
$(\forall i\in\II)$ $\prox_{g_i}(L_ix)=q_i$.
\item
\label{p:2i}
If $\beta\geq 1$, then 
\begin{equation}
F=\beta\,\prox_f,\quad\text{where}\quad
f=\Bigg(\sum_{i\in\II}\bigg(g_i^*\infconv
\dfrac{\|\cdot\|_{\GG_i}^2}{2}\bigg)
\circ L_i\Bigg)^*-\dfrac{\|\cdot\|_{\HH}^2}{2}.
\end{equation}
\end{enumerate}
\end{proposition}
\begin{proof}
\ref{p:2ii}--\ref{p:2iii}:
Apply Proposition~\ref{p:1} with $(\forall i\in\II)$
$Q_i=\prox_{g_i}$ and $\beta_i=1$. 

\ref{p:2i}: This follows from \cite[Proposition~3.9]{Bord18}.
\end{proof}

\begin{example}[scalar observations]
\label{ex:1}
We specialize the setting of Proposition~\ref{p:2} by assuming 
that, for some $i\in\II$, $\GG_i=\RR$ and $L_i=\scal{\cdot}{a_i}$,
where $0\neq a_i\in\HH$. Let us denote by 
$\chi_i=\prox_{g_i}\scal{\overline{x}}{a_i}$ the resulting
observation. This scenario allows us to recover various
nonlinear observation processes used in the literature. 
\begin{enumerate}
\item
\label{ex:1a}
Set $g_i=\iota_D$, where $D$ is a nonempty closed interval in 
$\RR$ with $\underline{\delta}=\inf D\in\RLX$ and 
$\overline{\delta}=\sup D\in\RX$. Then we obtain the hard 
clipping process
\begin{equation}
\chi_i=\proj_D\scal{\overline{x}}{a_i}=
\begin{cases}
\overline{\delta},&\text{if}\;\:\scal{\overline{x}}{a_i}>
\overline{\delta};\\
\scal{\overline{x}}{a_i},&\text{if}\;\:
\scal{\overline{x}}{a_i}\in D;\\
\underline{\delta},&\text{if}\;\:
\scal{\overline{x}}{a_i}<\underline{\delta},
\end{cases}
\end{equation}
which shows up in several nonlinear data collection processes;
see for instance \cite{Avil17,Fouc17,Stud12,Tesh19}. It models the
inability of the sensors to record values above $\overline{\delta}$
and below $\underline{\delta}$.
\item
\label{ex:1z}
Let $\Omega$ be a nonempty closed interval of $\RR$ and let 
$\soft{\Omega}$ be the associated soft thresholder, i.e.,
\begin{equation}
\label{e:soft4}
\soft{\Omega}\colon\RR\to\RR\colon\xi\mapsto
\begin{cases}
\xi-\overline{\omega},&\text{if}\;\:\xi>\overline{\omega};\\
0,&\text{if}\;\:\xi\in\Omega;\\
\xi-\underline{\omega},&\text{if}\;\:\xi<\underline{\omega},
\end{cases}
\quad
\qquad\text{with}\quad
\begin{cases}
\overline{\omega}=\sup\Omega\\
\underline{\omega}=\inf\Omega.
\end{cases}
\end{equation}
Further, let $\psi\in\Gamma_0(\RR)$ be differentiable at
$0$ with $\psi'(0)=0$, and set $g_i=\psi+\sigma_\Omega$, where
$\sigma_\Omega$ is the support function of $\Omega$.
Then it follows from \cite[Proposition~3.6]{Siop07} that 
\begin{equation}
\label{e:2007}
\chi_i=\prox_\psi\big(\soft{\Omega}\scal{\overline{x}}{a_i}\big)=
\begin{cases}
\prox_\psi(\scal{\overline{x}}{a_i}-\overline{\omega}),
&\text{if}\;\:\scal{\overline{x}}{a_i}>\overline{\omega};\\
0,&\text{if}\;\:\scal{\overline{x}}{a_i}\in\Omega;\\
\prox_\psi(\scal{\overline{x}}{a_i}-\underline{\omega}),
&\text{if}\;\:\scal{\overline{x}}{a_i}<\underline{\omega}.
\end{cases}
\end{equation}
In particular, if $\Omega=[-\omega,\omega]$ and $\psi=0$, we obtain
the standard soft thresholding operation 
\begin{equation}
\label{e:donoho}
\chi_i=\sign(\scal{\overline{x}}{a_i})
\max\{\abscal{\overline{x}}{a_i}-\omega,0\}
\end{equation}
of \cite{Dono94}. On the other hand, if 
$\Omega=\left]\minf,\omega\right]$ and $\psi=0$, we obtain
a nonlinear sensor model from \cite{Krau17}.
\item
\label{ex:1m}
In \ref{ex:1z} suppose that $\psi=\iota_D$, where $D$ is as in
\ref{ex:1a} and contains $0$ in its interior. Then \eqref{e:2007}
becomes
\begin{equation}
\chi_i=
\begin{cases}
\overline{\delta},&\text{if}\;\:\scal{\overline{x}}{a_i}
\geq\overline{\delta}+\overline{\omega};\\
\scal{\overline{x}}{a_i}-\overline{\omega},
&\text{if}\;\:\overline{\omega}<\scal{\overline{x}}{a_i}
<\overline{\delta}+\overline{\omega};\\
0,&\text{if}\;\:\scal{\overline{x}}{a_i}\in\Omega;\\
\scal{\overline{x}}{a_i}-\underline{\omega},
&\text{if}\;\:\underline{\delta}+\underline{\omega}
<\scal{\overline{x}}{a_i}<\underline{\omega};\\
\underline{\delta},&\text{if}\;\:
\scal{\overline{x}}{a_i}\leq\underline{\delta}+\underline{\omega}.
\end{cases}
\end{equation}
This operation combines hard clipping and soft thresholding.
\item
\label{ex:1b}
Set 
\begin{equation}
g_i\colon\xi\mapsto
\begin{cases}
\dfrac{(1+\xi)\ln(1+\xi)+(1-\xi)\ln(1-\xi)-\xi^2}{2},
&\text{if}\;\:|\xi|<1;\\
\ln(2)-1/2,&\text{if}\;\:|\xi|=1;\\
\pinf,&\text{if}\;\:|\xi|>1.
\end{cases}
\end{equation}
Then it follows from \cite[Example~2.12]{Svva20} that 
$\chi_i=\tanh(\scal{\overline{x}}{a_i})$.
This soft clipping model is used in \cite{Avil17,Ende12}.
\item
\label{ex:1c}
Set 
\begin{equation}
g_i\colon\xi\mapsto
\begin{cases}
-\dfrac{2}{\pi}\ln\Big(\cos\Big(\Frac{\pi\xi}{2}\Big)\Big)-
\dfrac{\xi^2}{2},&\text{if}\;\:|\xi|<1;\\
\pinf,&\text{if}\;\:|\xi|\geq 1.
\end{cases}
\end{equation}
Then it follows from \cite[Example~2.11]{Svva20} that 
$\chi_i=(2/\pi)\arctan(\scal{\overline{x}}{a_i})$.
This soft clipping model appears in \cite{Avil17}.
\item
\label{ex:1d}
Set
\begin{align}
g_i\colon\xi\mapsto
\begin{cases}
-|\xi|-\ln(1-|\xi|)-{\xi^2}/{2},&\text{if}\;\:|\xi|<1;\\
\pinf,&\text{if}\;\:|\xi|\geq 1.
\end{cases}
\end{align}
Then it follows from \cite[Example~2.15]{Svva20} that 
$\chi_i={\scal{\overline{x}}{a_i}}/(1+\abscal{\overline{x}}{a_i})$.
This soft clipping model is found in \cite{Ende12,Marm20}.
\item
\label{ex:1e}
Set
\begin{equation}
\label{e:g6}
g_i\colon\xi\mapsto
\begin{cases}
|\xi|+(1-|\xi|)\ln\big|1-|\xi|\big|-{\xi^2}/{2},
&\text{if}\;\:|\xi|<1;\\
1/2,&\text{if}\;\:|\xi|=1;\\
\pinf,&\text{if}\;\:|\xi|>1.
\end{cases}
\end{equation}
For every 
$\xi\in\left]-1,1\right[=\dom g_i'=\ran\prox_{g_i}$,
we have $\xi+g_i'(\xi)=-\sign(\xi)\ln(1-|\xi|)$. Hence,
\begin{equation}
\label{e:d6}
\big(\Id+g_i'\big)^{-1}
=\prox_{g_i}\colon\xi\mapsto\sign(\xi)\big(1-\exp(-|\xi|)\big)
\end{equation}
and, therefore,
$\chi_i=\sign(\scal{\overline{x}}{a_i})
(1-\exp(-|\scal{\overline{x}}{a_i}|))$.
This distortion model is found in \cite[Section~10.6.3]{Tarr19}.
\item
\label{ex:1f}
Let $\eta_i\in\RPP$ and set 
\begin{equation}
g_i\colon\xi\mapsto\eta_i\xi+
\begin{cases}
\xi\ln(\xi)+(1-\xi)\ln(1-\xi)-\xi^2/2,&
\text{if}\;\:\xi\in\zeroun;\\
0,&\text{if}\;\:\xi=0;\\
-1/2,&\text{if}\;\:\xi=1;\\
\pinf,&\text{if}\;\:\xi\in\RR\smallsetminus[0,1].
\end{cases}
\end{equation}
Proceeding as in \ref{ex:1e}, we obtain
\begin{equation}
\label{e:82}
\chi_i=\frac{1}{1+\exp(\eta_i-\scal{\overline{x}}{a_i})},
\end{equation}
which is an encoding scheme used in \cite{Kond15}.
\end{enumerate}
\end{example}

\begin{example}
\label{ex:3}
In Proposition~\ref{p:2} suppose that, for some $i\in\II$,
$g_i=\phi_i\circ d_{D_i}$, where $\phi_i\in\Gamma_0(\RR)$ is even
with $\phi_i(0)=0$, and $D_i\subset\GG_i$ is nonempty, closed, and
convex. Then it follows from \cite[Proposition~2.1]{Nmtm09} that
$q_i$ is the nonlinear observation defined as follows:
\begin{enumerate}
\item
\label{ex:3i}
Suppose that $\phi_i=\iota_{\{0\}}$. Then 
\begin{equation}
\label{e:1w}
q_i=\proj_{D_i}(L_i\overline{x})
\end{equation}
captures several applications. Thus, if $\HH=\RR^N$ and
$D_i=\menge{(\xi_i)_{1\leq i\leq N}\in\RR^N}
{\xi_1\leq\cdots\leq\xi_N}$, then $q_i$ is the
best isotonic approximation to $L_i\overline{x}$ \cite{Dair20}.
On the other hand, if $D_i$ is the closed ball
with center $0$ and radius $\rho_i\in\RPP$, then \eqref{e:1w}
reduces to the hard saturation process
\begin{equation}
\label{e:1h}
q_i= 
\begin{cases}
\dfrac{\rho_i}{\|L_i\overline{x}\|}L_i\overline{x},
&\text{if}\;\:\|L_i\overline{x}\|>\rho_i;\\
L_i\overline{x},&\text{if}\;\:\|L_i\overline{x}\|\leq\rho_i,
\end{cases}
\end{equation}
which can be viewed as an infinite dimensional version of 
Example~\ref{ex:1}\ref{ex:1a}.
\item
\label{ex:3ii}
Suppose that $\phi_i\neq\iota_{\{0\}}$ and set 
$\rho_i=\max\partial\phi_i(0)$. Then 
\begin{equation}
\label{e:1b}
q_i= 
\begin{cases}
L_i\overline{x}+\Frac{\prox_{\phi_i^*}d_{D_i}
(L_i\overline{x})}{d_{D_i}(L_i\overline{x})}
\big(\proj_{D_i}(L_i\overline{x})-L_i\overline{x}\big),
&\text{if}\;\:d_{D_i}(L_i\overline{x})>\rho_i;\\
\proj_{D_i}(L_i\overline{x}),&\text{if}\;\:
d_{D_i}(L_i\overline{x})\leq\rho_i.
\end{cases}
\end{equation}
In particular, assume that $D_i=\{0\}$. Then
\eqref{e:1b} reduces to the abstract soft thresholding process
\begin{equation}
\label{e:1c}
q_i= 
\begin{cases}
L_i\overline{x}-\Frac{\prox_{\phi_i^*}\|L_i\overline{x}\|}
{\|L_i\overline{x}\|}\;L_i\overline{x},
&\text{if}\;\:\|L_i\overline{x}\|>\rho_i;\\
0,&\text{if}\;\:\|L_i\overline{x}\|\leq\rho_i,
\end{cases}
\end{equation}
which cannot record inputs with norm below a certain value.
Let us further specialize to the setting in which
$\phi_i=\rho_i|\cdot|$ with $\rho_i\in\RPP$. Then
$\phi_i^*=\iota_{[-\rho_i,\rho_i]}$,
$\partial\phi_i(0)=[-\rho_i,\rho_i]$, and \eqref{e:1c} becomes 
\begin{equation}
\label{e:1d}
q_i= 
\begin{cases}
\bigg(1-\Frac{\rho_i}{\|L_i\overline{x}\|}\bigg)
L_i\overline{x},
&\text{if}\;\:\|L_i\overline{x}\|>\rho_i;\\
0,&\text{if}\;\:\|L_i\overline{x}\|\leq\rho_i,
\end{cases}
\end{equation}
which can be viewed as an infinite dimensional version of
\eqref{e:donoho}.
\end{enumerate}
\end{example}

\subsection{Prescriptions derived from non-cocoercive operators}
\label{sec:23}

Here, we exemplify observation processes which are not cocoercive,
and possibly not even continuous, but that can still be represented
by proximal points relative to some firmly nonexpansive operator,
as required in Problem~\ref{prob:1}. The results in this section
constructively provide the proximal points and phrase the
evaluation of each firmly nonexpansive operator in terms of the
nonlinearity in the observation process.

\begin{example}
\label{ex:17}
In the spirit of the shrinkage ideas of Corollary~\ref{c:8} and
Example~\ref{ex:5}, a prescription involving more general
transformations $(\varrho_i)_{i\in\II}$ can be used to derive an
equivalent prescribed proximal point. Let us adopt the 
setting of Corollary~\ref{c:8}, except that
$(\varrho_i)_{i\in\II}$ are now arbitrary operators from $\RR$ to
$\RR$ such that, for some $\delta\in\RPP$, 
$\sup_{i\in\II}|\varrho_i|\leq\delta|\cdot|$. Since
\begin{equation}
\sum_{i\in\II}\big|\varrho_i(\scal{\overline{x}}{e_i})\big|^2
\leq\delta^2\sum_{i\in\II}\abscal{\overline{x}}{e_i}^2=
\delta^2\|\overline{x}\|^2<\pinf,
\end{equation}
the prescription 
$q=\sum_{i\in\II}\varrho_i(\scal{\overline{x}}{e_i})e_i$ is well
defined. While $q$ is not a proximal point in general, an
equivalent proximal point $p$ can be constructed from it in certain
instances. To illustrate this process, let us first compute 
$(\forall i\in\II)$ 
$\chi_i=\scal{q}{e_i}=\varrho_i(\scal{\overline{x}}{e_i})$.
In both examples to follow, for every $i\in\II$, we construct an
operator $\sigma_i\colon\RR\to\RR$ such that
$\varphi_i=\sigma_i\circ\varrho_i$ is firmly nonexpansive,
$\varphi_i(0)=0$, and no information is lost when $\sigma_i$
is applied to the prescription 
$\chi_i=\varrho_i(\scal{\overline{x}}{e_i})$ in the sense that
\begin{equation}
\label{e:rt1}
(\forall\xi\in\RR)\quad 
\big[\,\chi_i=\varrho_i(\xi)\quad\Leftrightarrow\quad
\sigma_i(\chi_i)=\sigma_i\big(\varrho_i(\xi)\big)=\varphi_i(\xi)
\,\big].
\end{equation}
Using Corollary~\ref{c:8} with the firmly nonexpansive
operators $(\varphi_i)_{i\in\II}$, this implies that
$p=\sum_{i\in\II}\sigma_i(\chi_i)e_i$ is a proximal point of
$\overline{x}$.
\begin{enumerate}
\item
\label{ex:17i}
Let $i\in\II$, let $\omega_i\in\RPP$, and consider the
non-Lipschitzian sampling operator \cite{Abra98,Taov00} 
\begin{equation}
\label{e:22-1}
\varrho_i\colon\xi\mapsto
\begin{cases}
\sign(\xi)\sqrt{\xi^2-\omega_i^2},
&\text{if}\;\:|\xi|>\omega_i;\\
0,&\text{if}\;\:|\xi|\leq\omega_i.
\end{cases} 
\end{equation}
It is straightforward to verify that \eqref{e:rt1} holds with
\begin{equation}
\sigma_i\colon\xi\mapsto\sign(\xi)
\left(\sqrt{\xi^2+\omega_i^2}-\omega_i\right),
\end{equation}
in which case $\varphi_i=\sigma_i\circ\varrho_i$ is the
soft thresholder on $[-\omega_i,\omega_i]$ of \eqref{e:soft4}.
\item
\label{ex:17ii}
Let $i\in\II$, let $\omega_i\in\RPP$, and consider the
discontinuous sampling operator \cite{Taov00}
\begin{equation}
\varrho_i=\hard{[-\omega_i,\omega_i]}\colon\xi\mapsto
\begin{cases}
\xi,&\text{if}\;\:|\xi|>\omega_i;\\
0,&\text{if}\;\:|\xi|\leq\omega_i,
\end{cases}
\end{equation}
which is also known as the hard thresholder on
$[-\omega_i,\omega_i]$. This operator is used as a sensing model in
\cite{Boch13} and as a compression model in \cite{Teml98}. Then
\eqref{e:rt1} is satisfied with
\begin{equation}
\sigma_i\colon\xi\mapsto\xi-\omega_i\sign(\xi),
\end{equation}
in which case $\varphi_i=\sigma_i\circ\varrho_i$ turns out to be
the soft thresholder on
$[-\omega_i,\omega_i]$ of \eqref{e:soft4}.
\end{enumerate}
\end{example}

Next, we revisit Proposition~\ref{p:7} by relaxing the firm
nonexpansiveness of the observation operators and constructing an
equivalent proximal point via some transformation. This equivalence
is expressed in \ref{p:6i} below.

\begin{proposition}
\label{p:6}
Let $(\HS_i)_{i\in\II}$ be an at most countable family of real
Hilbert spaces, let $\HH=\bigoplus_{i\in\II}\HS_i$, let
$\overline{x}\in\HH$, and let 
$(\overline{\mathsf{x}}_i)_{i\in\II}$ be its decomposition, i.e.,
$(\forall i\in\II)$ $\overline{\mathsf{x}}_i\in\HS_i$. In addition,
for every $i\in\II$, let $\mathsf{Q}_i\colon\HS_i\to\HS_i$ and let
$\mathsf{q}_i=\mathsf{Q}_i\overline{\mathsf{x}}_i$. Suppose that
there exist operators $(\mathsf{S}_i)_{i\in\II}$ from $\HS_i$ to
$\HS_i$ such that the operators $(\mathsf{F}_i)_{i\in\II}=
(\mathsf{S}_i\circ\mathsf{Q}_i)_{i\in\II}$ satisfy the following:
\begin{enumerate}
\item
\label{p:6iii}
The operators $(\mathsf{F}_i)_{i\in\II}$ are firmly nonexpansive.
\item 
\label{p:6ii}
If $\II$ is infinite, there exists $(\mathsf{z}_i)_{i\in\II}\in\HH$
such that
$\sum_{i\in\II}\|\mathsf{F}_i\mathsf{z}_i-\mathsf{z}_i\|^2<\pinf$.
\item
\label{p:6i}
$(\forall i\in\II)(\forall\,\mathsf{x}_i\in\HS_i)$ 
$\big[\,\mathsf{F}_i\mathsf{x}_i=
\mathsf{S}_i\mathsf{q}_i$ $\Leftrightarrow$ 
$\mathsf{Q}_i\mathsf{x}_i=\mathsf{q}_i\,\big]$.
\end{enumerate}
Then $p=(\mathsf{S}_i\mathsf{q}_i)_{i\in\II}$ is the proximal point
of $\overline{x}$ relative to
$F\colon\HH\to\HH\colon(\mathsf{x}_i)_{i\in\II}\mapsto
(\mathsf{F}_i\mathsf{x}_i)_{i\in\II}$.
\end{proposition}
\begin{proof}
This follows from Proposition~\ref{p:7}.
\end{proof}

The following result illustrates the process described in
Proposition~\ref{p:6}, through a generalization of the
discontinuous hard thresholding operator of
Example~\ref{ex:17}\ref{ex:17ii}, which corresponds to the case
when $\HS_i=\RR$ and $\mathsf{C}_i=\{\mathsf{0}\}$ in \eqref{e:g1}
below.

\begin{proposition}
\label{p:g}
Let $(\HS_i)_{i\in\II}$ be an at most countable family of real
Hilbert spaces, let $\HH=\bigoplus_{i\in\II}\HS_i$, let
$\overline{x}\in\HH$, and let 
$(\overline{\mathsf{x}}_i)_{i\in\II}$ be its decomposition.
For every $i\in\II$, let $\omega_i\in\RPP$, let $\mathsf{C}_i$ be 
a nonempty closed convex subset of $\HS_i$, set 
\begin{equation}
\label{e:g1}
\mathsf{Q}_i\colon\HS_i\to\HS_i\colon\mathsf{x}_i\mapsto
\begin{cases}
\mathsf{x}_i,&\text{if}\;\:d_{\mathsf{C}_i}(\mathsf{x}_i)>
\omega_i;\\
\proj_{\mathsf{C}_i}\mathsf{x}_i,&\text{if}\;\:
d_{\mathsf{C}_i}(\mathsf{x}_i)\leq\omega_i,
\end{cases}
\end{equation}
and let $\mathsf{q}_i=\mathsf{Q}_i\overline{\mathsf{x}}_i$ be the
associated prescription. If $\II$ is infinite, suppose that
$(\forall i\in\II)$ $\mathsf{0}\in\mathsf{C}_i$. Further, for
every $i\in\II$, set
\begin{equation}
\label{e:g3}
\mathsf{S}_i\colon\HS_i\to\HS_i\colon \mathsf{x}_i\mapsto
\begin{cases}
\mathsf{x}_i+\dfrac{\omega_i}{d_{\mathsf{C}_i}(\mathsf{x}_i)}
(\proj_{\mathsf{C}_i}\mathsf{x}_i-\mathsf{x}_i),&\text{if}\;\:
\mathsf{x}_i\not\in\mathsf{C}_i;\\
\mathsf{x}_i,&\text{if}\;\:\mathsf{x}_i\in\mathsf{C}_i
\end{cases}
\qquad\text{and}\quad 
\begin{cases}
\mathsf{F}_i=\mathsf{S}_i\circ\mathsf{Q}_i\\
\mathsf{p}_i=\mathsf{S}_i\mathsf{q}_i.
\end{cases}
\end{equation}
Finally, set $p=(\mathsf{p}_i)_{i\in\II}$ and
$f\colon\HH\to\RX\colon(\mathsf{x}_i)_{i\in\II}\mapsto
\sum_{i\in\II}\omega_id_{\mathsf{C}_i}(\mathsf{x}_i)$. Then the
following hold:
\begin{enumerate}
\item
\label{p:gi}
For every $i\in\II$,
$\mathsf{F}_i=\prox_{\omega_i d_{\mathsf{C}_i}}$.
\item
\label{p:gii}
$p$ is the proximal point of $\overline{x}$ relative to $f$.
\item
\label{p:giii}
Let $x=(\mathsf{x}_i)_{i\in\II}\in\HH$. Then 
$\big[\,(\forall i\in\II)$ $\mathsf{Q}_i\mathsf{x}_i=
\mathsf{q}_i\,\big]$ $\Leftrightarrow$ $\prox_fx=p$.
\end{enumerate}
\end{proposition}
\begin{proof}
We derive from \eqref{e:g1}, \eqref{e:g3}, and 
\cite[Proposition~3.21]{Livre1} that
\begin{equation}
\label{e:60}
(\forall i\in\II)(\forall\mathsf{x}_i\in\HS_i)\quad
\mathsf{F}_i\mathsf{x}_i=
\begin{cases}
\proj_{\mathsf{C}_i}\mathsf{x}_i+
\bigg(1-\dfrac{\omega_i}{d_{\mathsf{C}_i}(\mathsf{x}_i)}\bigg)
\big(\mathsf{x}_i-\proj_{\mathsf{C}_i}\mathsf{x}_i\big)
\notin\mathsf{C}_i,&\text{if}\;\:
d_{\mathsf{C}_i}(\mathsf{x}_i)>\omega_i;\\
\proj_{\mathsf{C}_i}\mathsf{x}_i\in\mathsf{C}_i,&\text{if}\;\:
d_{\mathsf{C}_i}(\mathsf{x}_i)\leq\omega_i.
\end{cases}
\end{equation}

\ref{p:gi}: This is a consequence of \eqref{e:60} and
\cite[Example~24.28]{Livre1}.

\ref{p:gii}: If $\II$ is infinite, $(\forall i\in\II)$
$\mathsf{0}\in\mathsf{C}_i$ $\Rightarrow$
$d_{\mathsf{C}_i}(\mathsf{0})=0$ $\Rightarrow$
$\mathsf{F}_i(\mathsf{0})=\mathsf{0}$ by \eqref{e:60}. In turn, the
claim follows from Corollary~\ref{c:7} and \ref{p:gi}.

\ref{p:giii}: 
We first note that Corollary~\ref{c:7} and \ref{p:gi} imply that
\begin{equation}
\label{e:66}
(\mathsf{F}_i\mathsf{x}_i)_{i\in\II}=
(\prox_{\omega_id_{\mathsf{C}_i}}\mathsf{x}_i)_{i\in\II}=
\prox_f x.
\end{equation}
Now, suppose that $(\forall i\in\II)$ 
$\mathsf{Q}_i\mathsf{x}_i=\mathsf{q}_i$. Then $(\forall i\in\II)$ 
$\mathsf{F}_i\mathsf{x}_i=\mathsf{S}_i(\mathsf{Q}_i\mathsf{x}_i)
=\mathsf{S}_i\mathsf{q}_i=\mathsf{p}_i$. In turn, \eqref{e:66} 
yields $\prox_f x=(\mathsf{F}_i\mathsf{x}_i)_{i\in\II}=p$.
Conversely, suppose that $\prox_fx=p$ and fix $i\in\II$. We derive
from \eqref{e:66} and \eqref{e:g3} that
\begin{equation}
\label{e:61}
\mathsf{F}_i\mathsf{x}_i=\mathsf{p}_i=
\mathsf{S}_i\mathsf{q}_i
=\mathsf{S}_i(\mathsf{Q}_i\overline{\mathsf{x}}_i)
=\mathsf{F}_i\overline{\mathsf{x}}_i.
\end{equation}
We must show that $\mathsf{Q}_i\mathsf{x}_i=\mathsf{q}_i$. It
follows from \eqref{e:g1}, \eqref{e:60}, and \eqref{e:61} that
\begin{equation}
d_{\mathsf{C}_i}(\mathsf{x}_i)\leq\omega_i
\;\:\Leftrightarrow\;\:
\mathsf{Q}_i\mathsf{x}_i=
\proj_{\mathsf{C}_i}\mathsf{x}_i=
\mathsf{F}_i\mathsf{x}_i=\mathsf{F}_i\overline{\mathsf{x}}_i\in
\mathsf{C}_i
\;\:\Rightarrow\;\:
\begin{cases}
d_{\mathsf{C}_i}(\overline{\mathsf{x}}_i)\leq\omega_i\\
\mathsf{Q}_i\mathsf{x}_i=
\proj_{\mathsf{C}_i}\overline{\mathsf{x}}_i=
\mathsf{Q}_i\overline{\mathsf{x}}_i=\mathsf{q}_i.
\end{cases}
\end{equation}
On the other hand, \eqref{e:g1} yields
\begin{equation}
\label{e:64}
d_{\mathsf{C}_i}(\mathsf{x}_i)>\omega_i\quad\Rightarrow\quad
\mathsf{Q}_i\mathsf{x}_i=\mathsf{x}_i,
\end{equation}
while \eqref{e:61} and \eqref{e:60} yield
\begin{eqnarray}
d_{\mathsf{C}_i}(\mathsf{x}_i)>\omega_i
&\Rightarrow&
\mathsf{p}_i=\mathsf{F}_i\overline{\mathsf{x}}_i=
\mathsf{F}_i\mathsf{x}_i=
\proj_{\mathsf{C}_i}\mathsf{x}_i+
\bigg(1-\dfrac{\omega_i}{d_{\mathsf{C}_i}(\mathsf{x}_i)}\bigg)
\big(\mathsf{x}_i-\proj_{\mathsf{C}_i}\mathsf{x}_i\big)
\notin\mathsf{C}_i
\label{e:62}\\
&\Rightarrow&
\mathsf{F}_i\overline{\mathsf{x}}_i=
\proj_{\mathsf{C}_i}\overline{\mathsf{x}}_i+
\bigg(1-\dfrac{\omega_i}{d_{\mathsf{C}_i}
(\overline{\mathsf{x}}_i)}\bigg)
\big(\overline{\mathsf{x}}_i-\proj_{\mathsf{C}_i}
\overline{\mathsf{x}}_i\big)\;\:\text{and}\;\:
d_{\mathsf{C}_i}(\overline{\mathsf{x}}_i)>\omega_i
\label{e:63}\\
&\Rightarrow&
\mathsf{q}_i=\mathsf{Q}_i\overline{\mathsf{x}}_i=
\overline{\mathsf{x}}_i.
\end{eqnarray}
Therefore, in view of \eqref{e:64}, it remains to show that 
$\overline{\mathsf{x}}_i=\mathsf{x}_i$. Set 
$\mathsf{r}_i=\proj_{\mathsf{C}_i}\mathsf{p}_i$.
We deduce from \eqref{e:62}, \eqref{e:63}, and
\cite[Proposition~3.21]{Livre1} that
$\mathsf{r}_i=\proj_{\mathsf{C}_i}\overline{\mathsf{x}}_i
=\proj_{\mathsf{C}_i}\mathsf{x}_i$. Thus, \eqref{e:62} and 
\eqref{e:63} yield
\begin{equation}
\mathsf{p}_i-\mathsf{r}_i=
\bigg(1-\dfrac{\omega_i}{\|\mathsf{x}_i-\mathsf{r}_i\|}\bigg)
(\mathsf{x}_i-\mathsf{r}_i)=
\bigg(1-\dfrac{\omega_i}{\|\overline{\mathsf{x}}_i-
\mathsf{r}_i\|}\bigg)(\overline{\mathsf{x}}_i-\mathsf{r}_i).
\end{equation}
Taking the norm of both sides yields 
$\|\mathsf{x}_i-\mathsf{r}_i\|=
\|\overline{\mathsf{x}}_i-\mathsf{r}_i\|$ and hence
$\overline{\mathsf{x}}_i=\mathsf{x}_i$.
\end{proof}

\section{A block-iterative extrapolated algorithm for 
best approximation}
\label{sec:3}

We propose a flexible algorithm to solve the following abstract
best approximation problem. This new algorithm, which is of
interest in its own right, will be specialized in
Section~\ref{sec:4} to the setting of Problem~\ref{prob:1}.

\begin{problem}
\label{prob:6}
Let $\HH$ be a real Hilbert space, let $(C_i)_{i\in I}$ be
an at most countable family of closed convex subsets of $\HH$ 
with nonempty intersection $C$, and let $x_0\in\HH$. The goal 
is to find $\proj_Cx_0$, i.e., to
\begin{equation}
\label{e:06}
\text{minimize}\:\;\|x-x_0\|\quad\text{subject to}\quad
x\in\bigcap_{i\in I}C_i.
\end{equation}
\end{problem}

In 1968, Yves Haugazeau proposed in his unpublished thesis
\cite{Haug68} an iterative method to solve Problem~\ref{prob:6}
when $I$ is finite. His algorithm proceeds by periodic projections
onto the individual sets.

\begin{proposition}{\rm\cite[Th\'eor\`eme~3-2]{Haug68}}
\label{p:yves}
In Problem~\ref{prob:6}, suppose that $I$ is finite, say 
$I=\{0,\ldots,m-1\}$, where $2\leq m\in\NN$. Given 
$(s,t)\in\HH^2$ such that 
\begin{equation}
\label{e:yh2}
D=\menge{x\in\HH}{\scal{x-s}{x_0-s}\leq 0\;\:\text{and}\;\:
\scal{x-t}{s-t}\leq 0}\neq\emp,
\end{equation}
set $\chi=\scal{x_0-s}{s-t}$, $\mu=\|x_0-s\|^2$, $\nu=\|s-t\|^2$, 
and $\rho=\mu\nu-\chi^2$, and define
\begin{equation}
\label{e:yh3}
Q(x_0,s,t)=\proj_Dx_0=
\begin{cases}
t,&\text{if}\;\:\rho=0\;\:\text{and}\;\:\chi\geq 0;\\
x_0+\bigg(1+\dfrac{\chi}{\nu}\bigg)(t-s),&\text{if}\;\:\rho>0\;\:
\text{and}\;\:\chi\nu\geq\rho;\\[3mm]
s+\dfrac{\nu}{\rho}\big(\chi(x_0-s)+\mu(t-s)\big),&\text{if}\;\:
\rho>0\;\:\text{and}\;\:\chi\nu<\rho.
\end{cases}
\end{equation}
Construct a sequence $(x_n)_{n\in\NN}$ by iterating
\begin{equation}
\label{e:yh1}
\begin{array}{l}
\text{for}\;\:n=0,1,\ldots\\
\left\lfloor
\begin{array}{l}
t_n=\proj_{C_{n({\text{\rm mod}\:m)}}}x_n\\
x_{n+1}=Q(x_0,x_n,t_n).
\end{array}
\right.
\end{array}
\end{equation}
Then $x_n\to\proj_Cx_0$.
\end{proposition}

Haugazeau's algorithm uses only one set at each iteration. The
following variant due to Guy Pierra uses all of them
simultaneously.

\begin{proposition}{\rm\cite[Th\'eor\`eme~V.1]{Pier76}}
\label{p:guy}
In Problem~\ref{prob:6}, suppose that $I$ is finite, let $Q$
be as in Proposition~\ref{p:yves}, set $\omega=1/\card I$, and
fix $\varepsilon\in\zeroun$. Construct a sequence 
$(x_n)_{n\in\NN}$ by iterating
\begin{equation}
\label{e:gp1}
\begin{array}{l}
\text{for}\;\:n=0,1,\ldots\\
\left\lfloor
\begin{array}{l}
\text{for every}\;\:i\in I\\
\left\lfloor
\begin{array}{l}
a_{i,n}=\proj_{C_i}x_n\\
\theta_{i,n}=\|a_{i,n}-x_n\|^2\\
\end{array}
\right.\\
\theta_n=\omega\sum_{i\in I}\theta_{i,n}\\
\text{if}\;\theta_n=0\\
\left\lfloor
\begin{array}{l}
t_n=x_n\\
\end{array}
\right.\\
\text{else}\\
\left\lfloor
\begin{array}{l}
d_n=\omega\sum_{i\in I}a_{i,n}\\
y_n=d_n-x_n\\
\lambda_n=\theta_n/\|y_n\|^2\\
t_n=x_n+\lambda_ny_n\\
\end{array}
\right.\\[0.7mm]
x_{n+1}=Q(x_0,x_n,t_n).
\end{array}
\right.
\end{array}
\end{equation}
Then $x_n\to\proj_Cx_0$.
\end{proposition}

\begin{remark}
\label{r:gp}
An attractive feature of Pierra's algorithm \eqref{e:gp1} is
that, by convexity of $\|\cdot\|^2$, the relaxation parameter
$\lambda_n$ can extrapolate beyond $1$, hence attaining 
large values that induce fast convergence \cite{Sign03,Pier76}.
\end{remark}

Propositions~\ref{p:yves} and \ref{p:guy} were unified and extended
in \cite[Section~6.5]{Sico00} in the form of an algorithm for
solving Problem~\ref{prob:6} which is block-iterative in the sense
that, at iteration $n\in\NN$, only a subfamily of sets 
$(C_i)_{i\in I_n}$
needs to be activated, as opposed to all of them in \eqref{e:gp1}.
Block-iterative structures save time per iteration in two ways:
firstly, they do not require that every constraint be activated;
secondly, at every $n\in\NN$, activation of each constraint indexed
in $I_n$ can be performed in parallel and hence it is common to
select $\card I_n$ equal to the number of available processors.
Furthermore, in \cite[Section~6.5]{Sico00}, the sets 
$(C_i)_{i\in I}$ were specified as lower level sets of certain
functions and were activated by projections onto supersets instead
of exact ones as in \eqref{e:yh1} and \eqref{e:gp1}. 
Below, we propose an alternative block-iterative
scheme (Algorithm~\ref{algo:1}) which is more sophisticated in that
it leverages the affine structure of some sets $(C_i)_{i\in I'}$ to
produce deeper relaxation steps, hence providing extra acceleration
to the algorithm. Such affine-convex extrapolation techniques were
first discussed in \cite{Numa06}, where a weakly convergent method
was designed to solve convex feasibility problems, i.e., to find an
unspecified point in the intersection of closed convex sets.
Additionally, as will be seen in Section~\ref{sec:4}, this new
algorithm will be better suited to solve Problem~\ref{prob:1} to
the extent that it utilizes a fixed point model for the activation
of the sets. The following notions and facts lay the groundwork for
developing our best approximation algorithm.

\begin{definition}{\rm\cite[Section~4.1]{Livre1}}
\label{d:fqn}
$\mathfrak{T}$ is the class of \emph{firmly quasinonexpansive 
operators} from $\HH$ to $\HH$, i.e., 
\begin{equation}
\label{e:fqn}
\mathfrak{T}=\menge{T\colon\HH\to\HH}{(\forall x\in\HH)
(\forall y\in\Fix T)\;\:\scal{y-Tx}{x-Tx}\leq 0}.
\end{equation}
\end{definition}

\begin{example}{\rm\cite{Moor01,Livre1}} 
\label{ex:01}
Let $T\colon\HH\to\HH$ and set $C=\Fix T$. Then $T\in\mathfrak{T}$
in each of the following cases: 
\begin{enumerate}
\item
$T$ is the projector onto a nonempty closed convex subset $C$ of 
$\HH$.
\item
$T$ is the proximity operator of a function $f\in\Gamma_0(\HH)$.
Then $C=\Argmin f$.
\item
$T$ is the resolvent of a maximally monotone operator
$A\colon\HH\to 2^{\HH}$. Then $C=\menge{x\in\HH}{0\in Ax}$ is the
set of zeros of $A$.
\item
$T$ is firmly nonexpansive.
\item
$R=2T-\Id$ is quasinonexpansive:
$(\forall x\in\HH)(\forall y\in\Fix R)$ 
$\|Rx-y\|\leq\|x-y\|$. Then $C=\Fix R$.
\item
$T$ is a subgradient projector onto the lower level set 
$C=\menge{x\in\HH}{f(x)\leq 0}\neq\emp$ of a continuous convex
function $f\colon\HH\to\RR$, that is, given a selection $s$ of the
subdifferential of $f$,
\begin{equation}
\label{e:sproj}
(\forall x\in\HH)\;\:Tx=\sproj_C x=
\begin{cases}
x-\displaystyle{\frac{f(x)}{\|s(x)\|^2}}s(x),
&\;\text{if}\;\:f(x)>0;\\
x,&\;\text{if}\;\:f(x)\leq 0.
\end{cases}
\end{equation}
\end{enumerate}
\end{example}

\begin{lemma}{\rm\cite{Moor01,Livre1}} 
\label{l:12}
Let $T\colon\HH\to\HH$. If $T\in\mathfrak{T}$, then $\Fix T$ is
closed and convex. Conversely, if $C$ is a nonempty closed convex
subset of $\HH$, then $C=\Fix T$, where $T=\proj_C\in\mathfrak{T}$.
\end{lemma}

\begin{lemma}
\label{l:ws}
Let $(T_n)_{n\in\NN}$ be a sequence of operators in 
$\mathfrak{T}$ such that 
$\emp\neq C\subset\bigcap_{n\in\NN}\Fix T_n$, let
$x_0\in\HH$, let $Q$ be as in Proposition~\ref{p:yves}, and for
every $n\in\NN$, set $x_{n+1}=Q(x_0,x_n,T_nx_n)$. 
Then the following hold:
\begin{enumerate}
\item
\label{l:wsi}
$(x_n)_{n\in\NN}$ is well defined.
\item
\label{l:wsii}
$\sum_{n\in\NN}\|x_{n+1}-x_n\|^2<+\infty$.
\item
\label{l:wsiii}
$\sum_{n\in\NN}\|T_nx_n-x_n\|^2<+\infty$.
\item
\label{l:wsiv}
$x_n\to\proj_{C}x_0$ if and only if all the weak sequential 
cluster points of $(x_n)_{n\in\NN}$ lie in $C$.
\end{enumerate}
\end{lemma}
\begin{proof}
In the case when $\emp\neq C=\bigcap_{n\in\NN}\Fix T_n$, the
results are shown in 
\cite[Proposition~3.4(v) and Theorem~3.5]{Moor01}. However, an
inspection of these proofs reveals that they remain true in our
context.
\end{proof}

We are now in a position to introduce our best approximation
algorithm for solving Problem~\ref{prob:6}. It incorporates
ingredients of the best approximation method of
\cite[Section~6.5]{Sico00} and of the convex feasibility method of
\cite{Numa06}. 

\begin{algorithm}
\label{algo:1}
Consider the setting of Problem~\ref{prob:6} and 
denote by $(C_i)_{i\in I'}$ a subfamily of 
$(C_i)_{i\in I}$ of closed affine subspaces the projectors onto
which are easy to implement; this subfamily is assumed to be
nonempty as $\HH$ can be included in it. Let $Q$ be as in
Proposition~\ref{p:yves}, fix $\varepsilon\in\zeroun$, and iterate
\begin{equation}
\label{e:algo1}
\hskip -0.6mm
\begin{array}{l}
\text{for}\;\:n=0,1,\ldots\\
\left\lfloor
\begin{array}{l}
\text{take}\;\:\ii(n)\in I'\\
z_n=\proj_{C_{\ii(n)}}x_n\\
\text{take a nonempty finite set}\;\:I_n\subset I\\
\text{for every}\;\:i\in I_n\\
\left\lfloor
\begin{array}{l}
\text{take}\;\:T_{i,n}\in\mathfrak{T}\;\:\text{such that}\;\:
\Fix T_{i,n}=C_i\\
a_{i,n}=T_{i,n}z_n\\
\theta_{i,n}=\|a_{i,n}-z_n\|^2\\
\end{array}
\right.\\
\text{take}\;\:j_n\in I_n\;\:\text{such that}\;\:
\theta_{j_n,n}=\text{\rm max}_{i\in I_n}\theta_{i,n}\\
\text{take}\;\:\{\omega_{i,n}\}_{i\in I_n}\subset[0,1]\;\:
\text{such that}\;\:\sum_{i\in I_n}\omega_{i,n}=1
\;\:\text{and}\;\:\omega_{j_n,n}\geq\varepsilon\\
I_n^+=\menge{i\in I_n}{\omega_{i,n}>0}\\
\theta_n=\sum_{i\in I_n^+}\omega_{i,n}\theta_{i,n}\\
\text{if}\;\theta_n=0\\
\left\lfloor
\begin{array}{l}
t_n=z_n\\
\end{array}
\right.\\
\text{else}\\
\left\lfloor
\begin{array}{l}
d_n=\sum_{i\in I_n^+}\omega_{i,n}a_{i,n}\\
y_n=\proj_{C_{\ii(n)}}d_n-z_n\\
\text{take}\;\:\lambda_n\in\big[\varepsilon\theta_n/\|d_n-z_n\|^2,
\theta_n/\|y_n\|^{2}\big]\\
t_n=z_n+\lambda_ny_n\\
\end{array}
\right.\\[0.7mm]
x_{n+1}=Q(x_0,x_n,t_n).
\end{array}
\right.\\
\end{array}
\end{equation}
\end{algorithm}

\begin{remark}
\label{r:24}
Let us highlight some special cases and features of 
Algorithm~\ref{algo:1}.
\begin{enumerate}
\item
\label{r:24i}
If the only closed affine subspace is $\HH$ then, for every
$n\in\NN$, $z_n=x_n$, and the resulting algorithm has a structure
similar to that of \cite[Section~6.5]{Sico00}, except that the
operators $(T_{i,n})_{i\in I_n}$ are chosen differently. In
particular, this setting captures \eqref{e:yh1} and \eqref{e:gp1}.
\item
\label{r:24iii}
Suppose that the last step of the algorithm at iteration $n\in\NN$
is replaced by $x_{n+1}=t_n$. Then we recover an instance of the
(weakly convergent) convex feasibility algorithm of \cite{Numa06}
to find an unspecified point in $C=\bigcap_{i\in I}C_i$.
\item
\label{r:24ii}
At iteration $n\in\NN$, a block of sets $(C_i)_{i\in I_n}$ is
selected and each of its elements is activated via a firmly
quasinonexpansive operator. Example~\ref{ex:01} provides various
options to choose these operators, depending on the nature of the
sets.
\item
\label{r:24iv}
If nontrivial affine sets are present then, at iteration 
$n\in\NN$, we have $z_n\neq x_n$ in general. Thus, as discussed in
\cite{Cens12} and its references in the context of feasibility
algorithms (see \ref{r:24iii}), the resulting step $t_n$ is 
larger than when $z_n=x_n$, which typically yields 
faster convergence. This point will be illustrated numerically
for our best approximation algorithm in Section~\ref{sec:5}.
\end{enumerate}
\end{remark}

We now establish the strong convergence of an arbitrary sequence
$(x_n)_{n\in\NN}$ generated by Algorithm~\ref{algo:1} to the
solution to Problem~\ref{prob:6}. The last component of the proof
relies on Lemma~\ref{l:ws}\ref{l:wsiv}, i.e., showing that the weak
sequential cluster points of $(x_n)_{n\in\NN}$ lie in $C$. The
same property is required in \cite[Theorem~3.3]{Numa06} to show the
weak convergence of the variant described in
Remark~\ref{r:24}\ref{r:24iii}. This parallels the weak-to-strong
convergence principle of \cite{Moor01}, namely the transformation
of weakly convergent feasibility methods into strongly convergent
best approximation methods.

\begin{theorem}
\label{t:1}
In the setting of Problem~\ref{prob:6}, let $(x_n)_{n\in\NN}$ be
generated by Algorithm \ref{algo:1}. Suppose that the following
hold:
\begin{enumerate}[label={\rm[\alph*]}]
\item
\label{t:1i}
There exist strictly positive integers $(M_i)_{i\in I}$ such that
\begin{equation}
\label{e:8}
(\forall i\in I)(\forall n\in\NN)\quad
i\in\bigcup_{l=n}^{n+M_i-1}\{\ii (l)\}\cup I_l.
\end{equation}
\item
\label{t:1ii}
For every $i\in I\smallsetminus I'$, every $x\in\HH$, and every
strictly increasing sequence $(r_n)_{n\in\NN}$ in $\NN$,
\begin{equation}
\label{e:fc}
\begin{cases}
i\in\bigcap_{n\in\NN}I_{r_n}\\
\proj_{C_{\ii(r_n)}}x_{r_n}\weakly x\\
T_{i,r_n}\big(\proj_{C_{\ii (r_n)}}x_{r_n}\big)
-\proj_{C_{\ii(r_n)}}x_{r_n}\to 0
\end{cases}
\qquad\Rightarrow\quad x\in C_i.
\end{equation}
\end{enumerate}
Then $x_n\to\proj_{C}x_0$.
\end{theorem}
\begin{proof}
Let us fix $n\in\NN$ temporarily. Define
\begin{equation}
\label{e:43}
L_n\colon\HH\to\RR\colon z\mapsto
\begin{cases}
\displaystyle{\frac{\sum_{i\in I_n^+}\omega_{i,n}\|T_{i,n}z-z\|^2}
{\big\|\sum_{i\in I_n^+}\omega_{i,n}T_{i,n}z-z\big\|^2}},
&\text{if}\:\;z\notin\bigcap_{i\in I_n^+}C_i;\\
1,&\text{if}\:\;z\in\bigcap_{i\in I_n^+}C_i
\end{cases}
\end{equation}
and
\begin{equation}
\label{e:44}
S_n\colon\HH\to\HH\colon z\mapsto
z+L_n(z)\Bigg(\sum_{i\in I_n^+}\omega_{i,n}T_{i,n}z-z\Bigg).
\end{equation}
We derive from \cite[Proposition~2.4]{Else01} that
$S_n\in\mathfrak{T}$ and $\Fix S_n=\bigcap_{i\in I_n^+}
\Fix T_{i,n}=\bigcap_{i\in I_n^+}C_i$. 
We also observe that
\begin{equation}
\label{e:ga1}
\theta_n=0\;\Leftrightarrow\;S_nz_n=z_n\;\Leftrightarrow\;
z_n\in\bigcap_{i\in I_n^+}C_i=\Fix S_n.
\end{equation}
Now define
\begin{equation}
\label{e:45}
K_n\colon\HH\to\RR\colon x\mapsto
\begin{cases}
\dfrac{\big\|S_n\big(\proj_{C_{\ii(n)}}x\big)-
\proj_{C_{\ii(n)}}x\big\|^2}{\big\|\proj_{C_{\ii(n)}}
\big(S_n\big(\proj_{C_{\ii(n)}}x\big)\big)-
\proj_{C_{\ii(n)}}x\big\|^2},
&\text{if}\:\;\proj_{C_{\ii(n)}}x\notin\bigcap_{i\in I_n^+}C_i;\\
1,&\text{if}\:\;\proj_{C_{\ii(n)}}x\in\bigcap_{i\in I_n^+}C_i
\end{cases}
\end{equation}
and
\begin{multline}
\label{e:t}
T_n\colon\HH\to\HH\colon x\mapsto
\proj_{C_{\ii(n)}}x+\gamma_n(x)\Big(\proj_{C_{\ii(n)}}\Big(
S_n\big(\proj_{C_{\ii(n)}}x\big)\Big)-\proj_{C_{\ii(n)}}x\Big),\\
\text{where}\quad\gamma_n(x)
\in\left[\varepsilon,K_n(x)\right].
\end{multline}
Then it follows from \cite[Theorem~2.8]{Numa06} that
$T_n\in\mathfrak{T}$ and 
\begin{equation}
\label{e:59}
\emp\neq C\subset C_{\ii(n)}\cap\bigcap_{i\in I_n^+}C_i=
C_{\ii(n)}\cap\Fix S_n=\Fix T_n. 
\end{equation}
If $\theta_n\neq 0$, 
using \eqref{e:algo1}, \eqref{e:44}, and the fact that
$\proj_{C_{\ii(n)}}$ is an affine operator 
\cite[Corollary~3.22(ii)]{Livre1}, we obtain
\begin{align}
\label{e:49}
\proj_{C_{\ii(n)}}\Big(S_n\big(\proj_{C_{\ii(n)}}x_n\big)\Big)-
\proj_{C_{\ii(n)}}x_n
&=\proj_{C_{\ii(n)}}(S_nz_n)-z_n\nonumber\\
&=\proj_{C_{\ii(n)}}\Big(\big(1-L_n(z_n)\big)z_n+L_n(z_n)d_n\Big)
-z_n\nonumber\\
&=\big(1-L_n(z_n)\big)\proj_{C_{\ii(n)}}z_n+
L_n(z_n)\proj_{C_{\ii(n)}}d_n-z_n\nonumber\\
&=L_n(z_n)\big(\proj_{C_{\ii(n)}}d_n-z_n\big)\nonumber\\
&=L_n(z_n)y_n
\end{align}
and, therefore,
\begin{equation}
\label{e:51}
\big\|\proj_{C_{\ii(n)}}(S_nz_n)-z_n\big\|=L_n(z_n)\|y_n\|.
\end{equation}
Hence, \eqref{e:45} and \eqref{e:44} yield
\begin{equation}
\label{e:56}
K_n(x_n)=
\begin{cases}
\dfrac{\|S_n(z_n)-z_n\|^2}
{\big\|\proj_{C_{\ii(n)}}\big(S_nz_n\big)-z_n\big\|^2}
=\dfrac{\|L_n(z_n)(d_n-z_n)\|^2}{\|L_n(z_n)y_n\|^2}
=\dfrac{\|d_n-z_n\|^2}{\|y_n\|^2},&\text{if}\;\:\theta_n\neq 0;\\ 
1,&\text{if}\;\:\theta_n=0.
\end{cases}
\end{equation}
At the same time, we derive from \eqref{e:43}, \eqref{e:algo1}, 
and \eqref{e:ga1} that
\begin{equation}
\label{e:46}
L_n(z_n)=
\begin{cases}
\dfrac{\theta_n}{\|d_n-z_n\|^2},&\text{if}\;\:\theta_n\neq 0;\\ 
1,&\text{if}\;\:\theta_n=0.
\end{cases}
\end{equation}
Altogether, it results from \eqref{e:t}, \eqref{e:56}, and
\eqref{e:46} that, if $\theta_n\neq 0$, 
\begin{equation}
\label{e:53}
\gamma_n(x_n)L_n(z_n)\in
\big[\varepsilon L_n(z_n),K_n(x_n)L_n(z_n)\big]=
\big[\varepsilon\theta_n/\|d_n-z_n\|^2,\theta_n/\|y_n\|^2\big]
\end{equation}
and, in view of \eqref{e:algo1}, we can therefore set
$\lambda_n=\gamma_n(x_n)L_n(z_n)$. Thus, it follows from
\eqref{e:algo1} and \eqref{e:49} that 
\begin{align}
\label{e:52}
\theta_n\neq 0\;\Rightarrow\;t_n
&=z_n+\lambda_ny_n\nonumber\\
&=z_n+\gamma_n(x_n)L_n(z_n)y_n\nonumber\\
&=\proj_{C_{\ii(n)}}x_n+\gamma_n(x_n)\Big(\proj_{C_{\ii(n)}}
\big(S_n\big(\proj_{C_{\ii(n)}}x_n\big)\big)-
\proj_{C_{\ii(n)}}x_n\Big)\nonumber\\
&=T_nx_n.
\end{align}
On the other hand, \eqref{e:algo1} and \eqref{e:ga1} yield
\begin{equation}
\label{e:ga2}
\theta_n=0\;\Rightarrow\;t_n=z_n=S_nz_n=T_nx_n.
\end{equation}
Combining \eqref{e:52} and \eqref{e:ga2}, we obtain 
\begin{equation}
\label{e:yh4}
x_{n+1}=Q(x_0,x_n,T_nx_n). 
\end{equation}
Turning back to \eqref{e:t} and \eqref{e:algo1}, 
we deduce from \cite[Corollary~3.22(i)]{Livre1} that
\begin{align}
\label{e:94}
\|T_nx_n-x_n\|^2
&=\big\|z_n-x_n+\gamma_n(x_n)\big(
\proj_{C_{\ii(n)}}(S_nz_n)-z_n\big)\big\|^2\nonumber\\
&=\|z_n-x_n\|^2
+2\gamma_n(x_n)\scal{\proj_{C_{\ii(n)}}x_n-x_n}
{\proj_{C_{\ii(n)}}(S_nz_n)-\proj_{C_{\ii(n)}}x_n}\nonumber\\
&\quad\;+|\gamma_n(x_n)|^2
\big\|\proj_{C_{\ii(n)}}(S_nz_n)-z_n\big\|^2
\nonumber\\
&=\|z_n-x_n\|^2+|\gamma_n(x_n)|^2
\big\|\proj_{C_{\ii(n)}}(S_nz_n)-z_n\big\|^2
\nonumber\\
&\geq\|z_n-x_n\|^2+\varepsilon^2
\big\|\proj_{C_{\ii(n)}}(S_nz_n)-z_n\big\|^2.
\end{align}
Since \eqref{e:59} implies that 
\begin{equation}
\label{e:54}
\emp\neq C\subset\bigcap_{n\in\NN}\Fix T_n,
\end{equation}
we derive from \eqref{e:yh4} and Lemma~\ref{l:ws}\ref{l:wsi} 
that $(x_n)_{n\in\NN}$ is well defined. Furthermore, \eqref{e:94}
and Lemma~\ref{l:ws}\ref{l:wsiii} guarantee that 
\begin{equation}
\label{e:xp}
\sum_{n\in\NN}\|z_n-x_n\|^2<+\infty
\end{equation}
and 
\begin{equation}
\label{e:prp}
\sum_{n\in\NN}\|\proj_{C_{\ii(n)}}(S_nz_n)-z_n\|^2<+\infty.
\end{equation}
Finally, in view of \eqref{e:54} and
Lemma~\ref{l:ws}\ref{l:wsiv}, to conclude the
proof, it is enough to show that all the weak sequential cluster
points of $(x_n)_{n\in\NN}$ lie in $C$. Since we have at our
disposal \ref{t:1i}, \ref{t:1ii}, \eqref{e:xp}, and \eqref{e:prp},
showing this inclusion can be done by following the same steps as
in the proof of \cite[Theorem~3.3(vi)]{Numa06}. 
\end{proof}

\begin{remark}
\label{r:12}
Condition \ref{t:1i} in Theorem~\ref{t:1} states that, for each
$i\in I$, the set $C_i$ should be involved at least once 
every $M_i$ iterations. Condition \ref{t:1ii} in 
Theorem~\ref{t:1} is discussed in \cite[Section~3.4]{Numa06}, 
where concrete scenarios that satisfy it are described. 
\end{remark}

\section{Fixed point model and algorithm for Problem~\ref{prob:1}}
\label{sec:4}
To solve Problem~\ref{prob:1}, we are going to reformulate it as an
instance of Problem~\ref{prob:6}. To this end, let us set
\begin{equation}
\label{e:24}
(\forall k\in K)\quad C_k=\menge{x\in\HH}{F_kx=p_k}
\quad\text{and}\quad T_k=p_k+\Id-F_k.
\end{equation}
Then it follows from \eqref{e:f} that
\begin{equation}
\label{e:25}
(\forall k\in K)\quad T_k\;\text{is firmly nonexpansive and}\;
\Fix T_k=C_k.
\end{equation}
We therefore deduce from Lemma~\ref{l:12} that $(C_k)_{k\in K}$ are
closed convex subsets of $\HH$. Thus, upon setting $I=J\cup K$, we
recast Problem~\ref{prob:1} is an instantiation of
Problem~\ref{prob:6}. This leads us to the following solution
method based on Algorithm~\ref{algo:1}.

\begin{proposition}
\label{p:5}
In the setting of Problem~\ref{prob:1}, let $Q$ be as in
Proposition~\ref{p:yves}, fix $\varepsilon\in\zeroun$, and 
denote by $(C_i)_{i\in I'}$ a subfamily of 
$(C_i)_{i\in J}$ of closed affine subspaces the projectors onto
which are easy to implement; this subfamily is assumed to be
nonempty as $\HH$ can be included in it. Iterate
\begin{equation}
\label{e:algo2}
\hskip -0.6mm
\begin{array}{l}
\text{for}\;\:n=0,1,\ldots\\
\left\lfloor
\begin{array}{l}
\text{take}\;\:\ii(n)\in I'\\
z_n=\proj_{C_{\ii(n)}}x_n\\
\text{take a nonempty finite set}\;\:I_n\subset J\cup K\\
\text{for every}\;\:i\in I_n\\
\left\lfloor
\begin{array}{l}
\text{if}\;\:i\in J\\
\left\lfloor
\begin{array}{l}
\text{take}\;\:T_{i,n}\in\mathfrak{T}\;\:\text{such that}\;\:
\Fix T_{i,n}=C_i\\
a_{i,n}=T_{i,n}z_n\\
\end{array}
\right.\\
\text{if}\;\:i\in K\\
\left\lfloor
\begin{array}{l}
a_{i,n}=p_i+z_n-F_iz_n\\
\end{array}
\right.\\
\theta_{i,n}=\|a_{i,n}-z_n\|^2\\
\end{array}
\right.\\
\text{take}\;\:j_n\in I_n\;\:\text{such that}\;\:
\theta_{j_n,n}=\text{\rm max}_{i\in I_n}\theta_{i,n}\\
\text{take}\;\:\{\omega_{i,n}\}_{i\in I_n}\subset[0,1]\;\:
\text{such that}\;\:\sum_{i\in I_n}\omega_{i,n}=1
\;\:\text{and}\;\:\omega_{j_n,n}\geq\varepsilon\\
I_n^+=\menge{i\in I_n}{\omega_{i,n}>0}\\
\theta_n=\sum_{i\in I_n^+}\omega_{i,n}\theta_{i,n}\\
\text{if}\;\theta_n=0\\
\left\lfloor
\begin{array}{l}
t_n=z_n\\
\end{array}
\right.\\
\text{else}\\
\left\lfloor
\begin{array}{l}
d_n=\sum_{i\in I_n^+}\omega_{i,n}a_{i,n}\\
y_n=\proj_{C_{\ii(n)}}d_n-z_n\\
\text{take}\;\:\lambda_n\in\big[\varepsilon\theta_n/\|d_n-z_n\|^2,
\theta_n/\|y_n\|^{2}\big]\\
t_n=z_n+\lambda_ny_n\\
\end{array}
\right.\\[0.7mm]
x_{n+1}=Q(x_0,x_n,t_n).
\end{array}
\right.\\
\end{array}
\end{equation}
Suppose that condition \ref{t:1i} in Theorem~\ref{t:1} holds with
$I=J\cup K$, as well as the following:
\begin{enumerate}[label={\rm[\alph*]}]
\setcounter{enumi}{2}
\item
\label{p:5ii}
For every $i\in J\smallsetminus I'$, every $x\in\HH$, and every
strictly increasing sequence $(r_n)_{n\in\NN}$ in $\NN$,
\eqref{e:fc} holds.
\end{enumerate}
Then $(x_n)_{n\in\NN}$ converges strongly to the solution to
Problem~\ref{prob:1}.
\end{proposition}
\begin{proof}
Let us bring into play \eqref{e:24} and \eqref{e:25}.
As discussed above, Problem~\ref{prob:1} is an instance of 
Problem~\ref{prob:6}, where $I=J\cup K$. Now set
\begin{equation}
\label{e:26}
(\forall k\in K)(\forall n\in\NN)\quad T_{k,n}=T_k=p_k+\Id-F_k. 
\end{equation}
Then \eqref{e:algo1} reduces to \eqref{e:algo2} and, in view of
condition \ref{p:5ii} above, to conclude via Theorem~\ref{t:1}, 
it suffices to check that condition \ref{t:1ii} in 
Theorem~\ref{t:1} holds for every $k\in K$. Towards this goal, 
let us fix $k\in K$ and a strictly increasing sequence
$(r_n)_{n\in\NN}$ in $\NN$ such that 
$k\in\bigcap_{n\in\NN}I_{r_n}$, and let us set 
$(\forall n\in\NN)$ $u_n=\proj_{C_{\ii(r_n)}}x_{r_n}$. 
Suppose that $u_n\weakly x\in\HH$ and that 
$T_{k,r_n}u_n-u_n\to 0$. Then \eqref{e:26} yields
$T_ku_n-u_n\to 0$ and, since $T_k$ is nonexpansive by \eqref{e:25},
it follows from Browder's demiclosedness principle
\cite[Corollary~4.28]{Livre1} that $x\in\Fix T_k=C_k$, which
concludes the proof.
\end{proof}

As was mentioned in Remark~\ref{r:24}\ref{r:24iv} and will be
illustrated in Section~\ref{sec:5}, exploiting the presence of
affine subspaces typically leads to faster convergence.
Problem~\ref{prob:1} can nonetheless be solved without taking the
affine subspaces into account. Formally, this amounts to
considering that $(C_i)_{i\in I'}$ consists solely of $\HH$, in
which case Proposition~\ref{p:5} leads to the following
implementation.

\begin{corollary}
\label{c:5}
In the setting of Problem~\ref{prob:1}, let $Q$ be as in
Proposition~\ref{p:yves}, and fix $\varepsilon\in\zeroun$. 
Iterate
\begin{equation}
\label{e:algo3}
\hskip -0.6mm
\begin{array}{l}
\text{for}\;\:n=0,1,\ldots\\
\left\lfloor
\begin{array}{l}
\text{take a nonempty finite set}\;\:I_n\subset J\cup K\\
\text{for every}\;\:i\in I_n\\
\left\lfloor
\begin{array}{l}
\text{if}\;\:i\in J\\
\left\lfloor
\begin{array}{l}
\text{take}\;\:T_{i,n}\in\mathfrak{T}\;\:\text{such that}\;\:
\Fix T_{i,n}=C_i\\
a_{i,n}=T_{i,n}x_n\\
\end{array}
\right.\\
\text{if}\;\:i\in K\\
\left\lfloor
\begin{array}{l}
a_{i,n}=p_i+x_n-F_ix_n\\
\end{array}
\right.\\
\theta_{i,n}=\|a_{i,n}-x_n\|^2\\
\end{array}
\right.\\
\text{take}\;\:j_n\in I_n\;\:\text{such that}\;\:
\theta_{j_n,n}=\text{\rm max}_{i\in I_n}\theta_{i,n}\\
\text{take}\;\:\{\omega_{i,n}\}_{i\in I_n}\subset[0,1]\;\:
\text{such that}\;\:\sum_{i\in I_n}\omega_{i,n}=1
\;\:\text{and}\;\:\omega_{j_n,n}\geq\varepsilon\\
I_n^+=\menge{i\in I_n}{\omega_{i,n}>0}\\
\theta_n=\sum_{i\in I_n^+}\omega_{i,n}\theta_{i,n}\\
\text{if}\;\theta_n=0\\
\left\lfloor
\begin{array}{l}
t_n=x_n\\
\end{array}
\right.\\
\text{else}\\
\left\lfloor
\begin{array}{l}
y_n=\sum_{i\in I_n^+}\omega_{i,n}a_{i,n}-x_n\\[1mm]
\text{take}\;\:\lambda_n\in\big[\varepsilon\theta_n/\|y_n\|^2,
\theta_n/\|y_n\|^{2}\big]\\
t_n=x_n+\lambda_ny_n\\
\end{array}
\right.\\[0.7mm]
x_{n+1}=Q(x_0,x_n,t_n).
\end{array}
\right.\\
\end{array}
\end{equation}
Suppose that the following hold:
\begin{enumerate}[label={\rm[\alph*]}]
\setcounter{enumi}{3}
\item
\label{c:5i}
There exist strictly positive integers $(M_i)_{i\in J\cup K}$ 
such that $(\forall i\in J\cup K)(\forall n\in\NN)$
$i\in\bigcup_{l=n}^{n+M_i-1} I_l$.
\item
\label{c:5ii}
For every $i\in J$, every $x\in\HH$, and every
strictly increasing sequence $(r_n)_{n\in\NN}$ in $\NN$,
\begin{equation}
\label{e:c5}
\bigg[\,i\in\bigcap_{n\in\NN}I_{r_n},\;
x_{r_n}\weakly x,\,\;\text{and}\;\;
T_{i,r_n}x_{r_n}-x_{r_n}\to 0\,\bigg]
\quad\Rightarrow\quad x\in C_i.
\end{equation}
\end{enumerate}
Then $(x_n)_{n\in\NN}$ converges strongly to the solution to
Problem~\ref{prob:1}.
\end{corollary}

\section{Numerical illustration}
\label{sec:5}
Let $\HH$ be the standard Euclidean space $\RR^N$, where $N=1024$.
The goal is to recover the original form of the signal
$\overline{x}\in\HH$ shown in Figure~\ref{fig:1} from the
following:
\begin{enumerate}
\item 
\label{ex:j1}
$\overline{x}$ resides in the subspace $C_1$ of signals which are
band-limited in the sense that their discrete Fourier transform 
vanishes outside of the $103$ lowest frequency components.
\item
\label{ex:j2}
Let $\tv\colon\HH\to\RR\colon x=(\xi_i)_{1\leq i\leq N}
\mapsto\sum_{1\leq i\leq N-1}|\xi_{i+1}-\xi_i|$ be the total
variation function. An upper bound $\gamma\in\RPP$ on
$\tv(\overline{x})$ is available. 
The associated constraint set is
$C_2=\menge{x\in\HH}{\tv(x)-\gamma\leq 0}$.
For this experiment, $\gamma=1.5\tv(\overline{x})$. 
\item 
\label{ex:j3}
$25$ observations $(q_k)_{k\in K}$ are available where, for every
$k\in K=\{3,\ldots,27\}$, $q_k$ is the isotonic regression of the
coefficients of
$\overline{x}$ in a dictionary $(e_{k,j})_{1\leq j\leq 10}$ of
vectors in $\HH$. More precisely (see
Example~\ref{ex:3}\ref{ex:3i}), set $\GG=\RR^{10}$ and 
$D=\menge{(\xi_j)_{1\leq j\leq 10}\in\GG}
{\xi_1\leq\cdots\leq\xi_{10}}$. Then, for every $k\in K$,
$q_k=\proj_D(L_k\overline{x})$, where
$L_k\colon\HH\to\GG\colon x\mapsto
(\scal{x}{e_{k,j}})_{1\leq j\leq 10}$.
\end{enumerate}
We seek the minimal-energy signal consistent with the information
above, i.e., we seek to 
\begin{equation}
\label{e:j}
\text{minimize}\:\;\|x\|\quad\text{subject to}\quad
x\in C_1\cap C_2\quad\text{and}\quad(\forall k\in K)\quad
\proj_D(L_kx)=q_k.
\end{equation}
Let us set $x_0=0$, $J=\{1,2\}$, and, for every $k\in K$,
$p_k=\|L_k\|^{-2}L_k^*q_k$, and
$F_k=\|L_k\|^{-2}L_k^*\circ\proj_D\circ L_k$. For every $k\in K$,
applying Proposition~\ref{p:1} with $\II=\{k\}$, $\GG_k=\GG$,
$\beta_k=1$, and $Q_k=\proj_D$ shows that $p_k$ is the proximal
point of $\overline{x}$ relative to $F_k$ and, for every $x\in\HH$,
$F_kx=p_k$ $\Leftrightarrow$ $\proj_D(L_kx)=q_k$. 
We therefore arrive at an instance of Problem~\ref{prob:1} which is
equivalent to \eqref{e:j}, namely
\begin{equation}
\label{e:j2}
\text{minimize}\:\;\|x\|\quad\text{subject to}\quad
x\in C_1\cap C_2\quad\text{and}\quad(\forall k\in K)\quad
F_kx=p_k.
\end{equation}
With an eye towards algorithm \eqref{e:algo2}, since $C_1$ is an
affine subspace with a straightforward projector \cite{Youl78}, 
set $I'=\{1\}$. At iteration $n\in\NN$, the
constraint \ref{ex:j2} is activated by the subgradient projector
$T_{2,n}=\sproj_{C_2}$ of \eqref{e:sproj} (see
\cite{Imag04} for its computation) since the direct projector is
hard to implement. The fact that 
condition \ref{p:5ii} in Proposition~\ref{p:5} is satisfied follows
from \cite[Proposition~29.41(vi)(a)]{Livre1}.
We solve \eqref{e:j2} with algorithm \eqref{e:algo2} to obtain the
solution $x_\infty$ shown in Figure~\ref{fig:2} (see
\cite[Algorithm~8.1.1]{Hard90} for the computation of 
$\proj_{D}$). 

To demonstrate the benefits of exploiting the presence of affine
subspaces in algorithm \eqref{e:algo2}, we show in
Figure~\ref{fig:3} the approximate solution it generates after
$1000$ iterations. For the sake of comparison, we display in 
Figure~\ref{fig:4} the approximate
solution generated by algorithm \eqref{e:algo3} after $1000$
iterations. The following parameters are used: 
\begin{itemize}
\item
{\bfseries Algorithm~\eqref{e:algo2}:}
For every $n\in\NN$, $\ii(n)=1$, 
and whenever $\theta_n\neq 0$, 
\begin{equation}
\label{e:39}
\lambda_n=
\begin{cases}
\dfrac{\theta_n}{2\|y_n\|^2},&\text{if}\;\:n\equiv 0\mod
3;\\[1.0em]
\dfrac{\theta_n}{\|y_n\|^2},&\text{if}\;\:n\not\equiv 0\mod 3.
\end{cases}
\end{equation}
Additionally, $I_n$ is selected to activate
$C_2$ at every iteration and periodically sweep through
one entry of $K$ per iteration, hence satisfying
condition~\ref{t:1i} in Theorem~\ref{t:1} with $M_1=M_2=1$,
and, for every $k\in K$, $M_k=25$. Moreover, 
for every $i\in I_n$, $\omega_{i,n}=1/2$.
\item
{\bfseries Algorithm~\eqref{e:algo3}:} Iteration $n\in\NN$ is
executed with the same relaxation scheme \eqref{e:39} as in
algorithm \eqref{e:algo2}, and the same choice of the activation
set $I_n$, with the exception that $I_n$ also activates $C_1$ at
every iteration. In addition, for every $i\in I_n$,
$\omega_{i,n}=1/3$. 
\end{itemize}

While both approaches are equivalent means of solving \eqref{e:j2},
Figures~\ref{fig:3} and \ref{fig:4} demonstrate qualitatively that
algorithm \eqref{e:algo2} yields faster convergence to the
solution $x_\infty$ than algorithm \eqref{e:algo3}. This is
confirmed quantitatively by the error plots of Figure~\ref{fig:5}.

\begin{figure}[H]
~\hspace{.05cm}

\caption{Original signal $\overline{x}$.}
\label{fig:1}
\end{figure}

\begin{figure}[H]
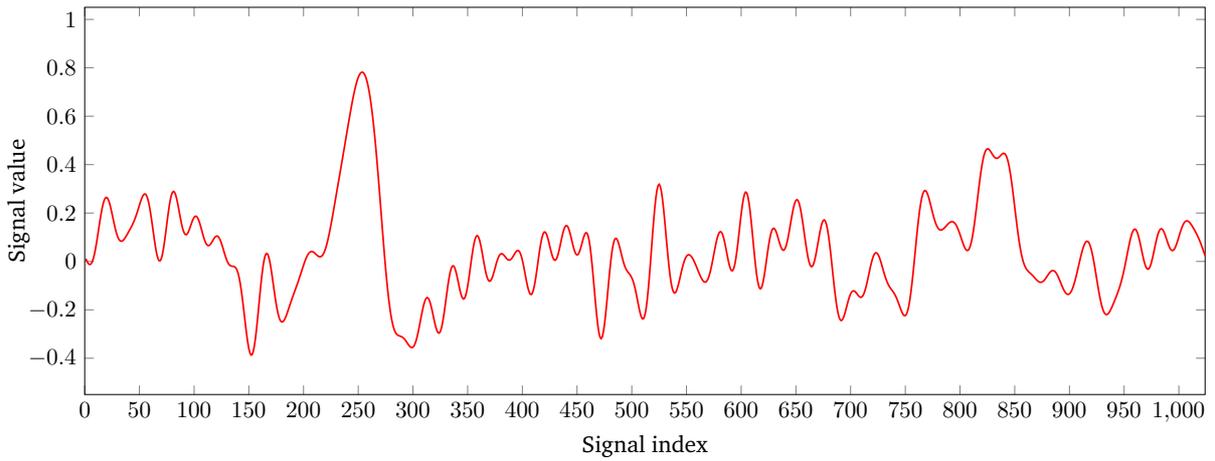

~\hspace{.05cm}
%
\caption{
Solution $x_{\infty}$ to \eqref{e:j}.}
\label{fig:2}
\end{figure} 

\begin{figure}[H]
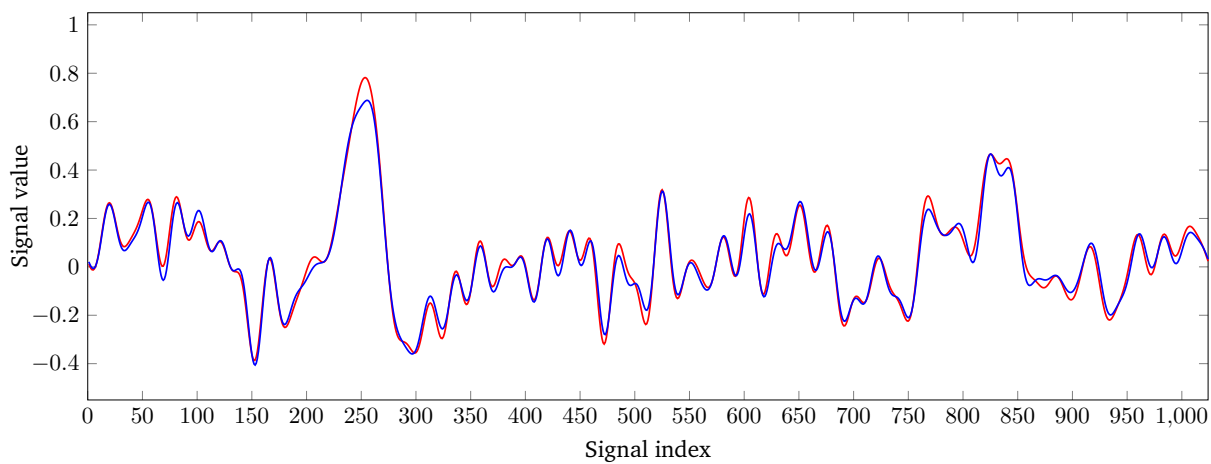

~\hspace{.05cm}
%
\caption{Solution $x_{\infty}$ (red) and the approximate recovery
obtained with $1000$ iterations of algorithm \eqref{e:algo2}, which
exploits affine constraints (blue).}
\label{fig:3}
\end{figure} 

\begin{figure}[H]
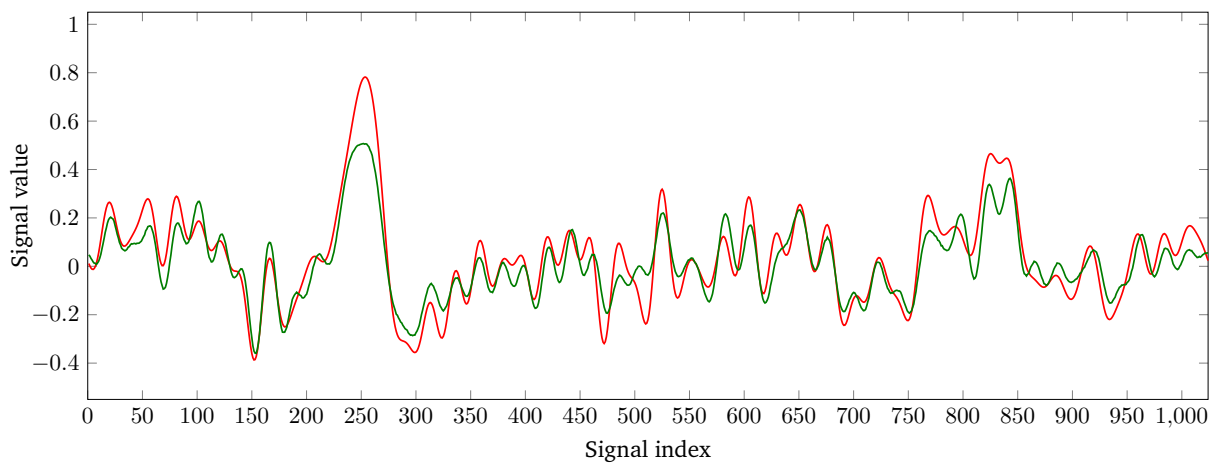

~\hspace{.05cm}
%
\caption{
Solution $x_{\infty}$ (red) and the approximate recovery obtained
with $1000$ iterations of algorithm \eqref{e:algo3}, which does not
exploit affine constraints (green).
}
\label{fig:4}
\end{figure} 

\begin{figure}[H]
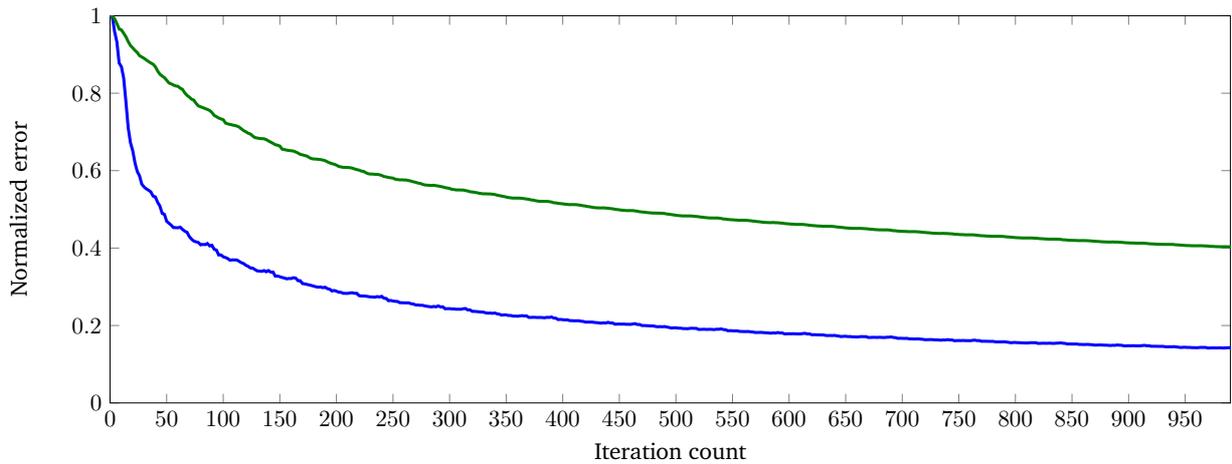

~\hspace{-.3cm}
%
\caption{Normalized error $\|x_n-x_{\infty}\|/\|x_0-x_{\infty}\|$
versus iteration count $n\in\{0,\dots,1000\}$ for algorithm
\eqref{e:algo2} (blue) and algorithm \eqref{e:algo3} (green).}
\label{fig:5}
\end{figure}

\end{document}